\documentclass[10pt,reqno]{amsart}
  \usepackage{amssymb}
  \usepackage{bm}
  \usepackage{geometry}
  \usepackage{mathrsfs}
  \usepackage{esvect}
  \usepackage{amsmath}
  \usepackage{hyperref}
  \usepackage{url}
\usepackage{array} 
\usepackage{varwidth}
  \usepackage{paralist}
    \usepackage{tikz-cd}
  \usepackage{tikz}
	  \usetikzlibrary{matrix,arrows}
      \tikzset{>=latex}
	\usepackage{amsthm}
      \newtheorem{theorem}{Theorem}
      
      \newtheorem*{theorem*}{Theorem}

\numberwithin{figure}{section}
\numberwithin{table}{section}

      \newtheorem{corollary}[theorem]{Corollary}

    \theoremstyle{definition}
      
      \newtheorem{definition}[theorem]{Definition}

      \newtheorem{proposition}[theorem]{Proposition}
      \newtheorem{construction}[theorem]{Construction}
      \newtheorem{lemma}[theorem]{Lemma}
      \newtheorem{remark}[theorem]{Remark}
\newtheorem*{dim2}{\textup{\textbf{Theorem ~\ref{theorem:MAIN(co)dim2-eff-neff}}}}
\newtheorem*{dim3}{\textup{\textbf{Theorem ~\ref{theorem:MAIN(co)dim3-eff-neff}}}}

\definecolor{mycolor}{RGB}{180,180,180}

	\numberwithin{equation}{section}
	\numberwithin{theorem}{section}

\newcommand{\bburl}[1]{\textcolor{blue}{\url{#1}}}
  \newcommand{\betarep}{b}

  \newcommand{\Hilb}{\textup{Hilb}}
  \newcommand{\Nef}{\textup{Nef}}
  \newcommand{\Eff}{\textup{Eff}}
  \newcommand{\MS}{\textup{MS}}
  
  \newcommand{\RL}{\textup{RL}}
  \newcommand{\wt}{\widetilde}

  \newcommand{\C}{{\mathbb C}}

\newcommand{\OO}{{\mathcal O}}

\newcommand{\Q}{{\mathbb Q}}

\newcommand{\Al}{\textup{Al}}
\newcommand{\Ker}{\textup{Ker}}

\newcommand{\col}{\textup{col}}
\newcommand{\nr}{\textup{nonred}}

\newcommand{\A}{{\mathbb A}}

\newcommand{\PP}{{\mathbb P}}

\begin{document}

\title{Higher codimension nef and effective cycles of $\Hilb_3 \PP^3$}
\author[Moreland]{Gwyneth Moreland}
\email{\textcolor{blue}{\href{mailto:gwynm@uic.edu}{gwynm@uic.edu}}}
\address{Department of Mathematics, Stat., \& CS, University of Illinois Chicago, Chicago, IL 60607}

\begin{abstract}
Nef and effective cones of divisors have been the subject of much study. In contrast, their higher codimension analogues are much harder to compute and few examples exist in the literature. In this paper we compute the nef cones in codimensions 2 \& 3 and the effective cones in dimensions 2 \& 3 for the Hilbert scheme of three points in $\PP^3$. Our computation generalizes results of Ryan \& Stathis and requires a careful analysis of the PGL orbits in the Hilbert scheme, as well as a new basis of the Chow ring inspired by Mallavibarrena and Sols.

\end{abstract}

\maketitle
\tableofcontents

\section{Introduction} 
 Let $\Hilb_m \PP^n$ denote the Hilbert scheme of zero-dimensional, length $m$ schemes in $\PP^n$. Much work has been done studying Hilbert schemes of points in the surface case (ex: \cite{ES87}, \cite{MS90}, \cite{Go90}, \cite{F95}, \cite{N96}, \cite{LQW04}, \cite{OP10}, \cite{ACBH13}, \cite{Hu16}, \cite{R16}, \cite{BS17},  \cite{A21}, \cite{RS21}, \cite{BZ23}) where these schemes are well-behaved. $\Hilb_m \PP^2$ is smooth, irreducible, reduced, and we have a birational map
\[\Hilb_m \PP^2 \rightarrow \textup{Sym}^m \PP^2.\]
Hence a generic $\Gamma \in \Hilb_m \PP^2$ corresponds to the union of $m$ distinct points. In higher dimensions, these schemes quickly become poorly behaved: $\Hilb_m \PP^3$ need not be irreducible \cite{I72} nor smooth. $\Hilb_4 \PP^3$, for example, is singular along the fat points. Jelisiejew \cite{J20} and Szachniewicz \cite{S21} demonstrate that Hilbert schemes of points in higher dimension need not be reduced.
$\Hilb_3 \PP^3$, however, is relatively well-behaved: it is irreducible, reduced, and smooth, and there are nice bases for the Chow/cohomology ring \cite{RL90}, \cite{FG93}. Nonetheless, we still get some of the interesting features of higher dimension, coming from the larger ranks of the Chow groups $A_k(\Hilb_3 \PP^3)$ and more complicated collinear loci than in the case of $\Hilb_3 \PP^2$.

In this paper, we look at higher (co)dimension nef and effective cones of $\Hilb_3 \PP^3$. Nef and effective cones in (co)dimension 1 have been a subject of study due to their connections to the minimal model program, as the nef cone of divisors is the closure of the ample cone. Higher (co)dimension analogs have been studied recently in \cite{DELV11}, \cite{CLO16}, \cite{CC15}, as well as \cite{RS21}. In particular, Ryan and Stathis \cite{RS21} study the (co)dimension 2 \& 3 nef and effective cones in $\Hilb_3 \PP^2$, utilizing the group action of $\textup{PGL}_3(\mathbb{C})$. While the study of positive cones in intermediate dimension is geometrically rich, we still have far fewer examples due to difficulties of working in intermediate dimension: the relevant vector spaces are often large in rank, we lack nice criteria for the cones being polyhedral, and some of the useful tools for curves and divisors no longer work. For example, a nef class need not be pseudoeffective in intermediate dimension. This paper provides illuminating new examples.

We will see that in the case of the dimension 2 and 3 effective cones and codimension 2 and 3 nef cones of $\Hilb_3 \PP^3$, these cones are polyhedral.

Our main theorems are stated below and proven in Section \ref{section:nef-eff-Hilb3P3}.

\begin{dim2} The effective cone of surfaces $\Eff_2(\Hilb_3 \PP^3)$ and the nef cone of sevenfolds $\Nef^2(\Hilb_3 \PP^3)$ are finite polyhedral cones generated by 7 and 8 generators respectively. We have explicit effective geometric representatives for each of the rays.
\end{dim2}
\begin{dim3} $\Eff_3(\Hilb_3 \PP^3)$ and $\Nef^3(\Hilb_3 \PP^3)$ are finite polyhedral cones with $13$ and $26$ generators respectively. We have explicit effective geometric representatives for each of the rays.
\end{dim3}

To prove Theorems \ref{theorem:MAIN(co)dim2-eff-neff} and \ref{theorem:MAIN(co)dim3-eff-neff} we first explicitly construct new bases of $A_i(\Hilb_3 \PP^3)$ for $i=2,3,6,7$. We then compute many of the positive cones of intermediate cycles for the PGL-orbit closures in $\Hilb_3 \PP^3$. We use a variant of Kleiman's transversality theorem to prove the nefness of certain classes. Our basis is especially suited to this application. The dual of the cone spanned by the nef classes we describe consists of effective classes by explicit construction, hence they span the effective cone. Since the nef classes we construct are all effective, one sees that the nef cone in codimensions $2$ and $3$ is contained in the effective cone-- a property not guaranteed in higher codimension.



\begin{remark}Since the dimension of the ambient vector space $A_2(\Hilb_3 \PP^n)$ stabilizes once $n \ge 4$, Theorem \ref{theorem:MAIN(co)dim2-eff-neff} gets us quite closer to determining the effective cone of surfaces for all $\Hilb_3 \PP^n$. This would provide a large class of examples of positive cones in intermediate dimension. One obstruction to solving the general case is that current methods rely on knowing intersection data about $A_k(\Hilb_3 \PP^n)$ for values of $k$ close to $n/2$. As $n$ grows these Chow groups become more unwieldy. It is worth noting that, from the proof of the $\Hilb_3 \PP^2$ and $\Hilb_3 \PP^3$ cases, we expect the effect cones of surfaces $\Eff_2(\Hilb_3 \PP^n)$ to stabilize as well once $n \ge 4$.
\end{remark}

\begin{remark} This leaves the effective and nef cones of fourfolds and fivefolds ($\Eff_4,\Eff_5,\Nef^4,\Nef^5$) as the remaining open computations for $\Hilb_3 \PP^3$. Understanding these would require better understanding the locus of nonreduced schemes in $\Hilb_3 \PP^3$. This is a difficult task: \textit{a priori} there are no good ways to get a handle on the structure of this locus.
\end{remark}

\begin{remark} The effective and nef cones of curves and divisors can be computed via standard methods of finding moving curves and contracted rays, similar to computations for $\Hilb_3 \PP^2$. We summarize the generators below; a proof can be found in \cite{ACBH13}. 

The effective cone of curves $\Eff_1(\Hilb_3 \PP^3)$ has two generators: the locus of schemes given by fixing two points on a line and allowing the third point to vary along the line, and the locus of scehemes whose general member is contained in a fixed plane, consisting of a fixed point $P$ and a length two scheme supported at another fixed point $Q$. The cone of nef divisors is then generated by the hyperplane class and the locus of schemes containing a length two subscheme coplanar with a fixed line $L$. 

The effective cone of divisors is generated by the nonreduced locus and the locus of schemes coplanar with a fixed point. The nef cone of curves is then generated by: the locus of schemes consisting of a fixed point $P$, a point allowed to vary on a fixed line $L$ containing $P$, and another fixed point $Q$ not on $L$, as well as the locus of schemes consisting of two fixed points $P,Q$, and a third point allowed to vary on a fixed line $L'$ not containing $P$ or $Q$.
\end{remark}

\begin{remark} This paper utilizes some intersection data computed for $\Hilb_3 \PP^3$ given in \cite{RL90}. We would like to give a word of caution to the reader: the intersection tables at the end of \cite{RL90} contain a few errors. 
  Four of the values are erroneous; the computation of their correct values is done in Lemma \ref{lem:corrections}. 
  \end{remark}

\begin{remark} Certain change-of-basis and intersection theory computations in this paper utilize Mathematica and Python; code for this can be found at:
\\
\url{github.com/gwynethmoreland/gwynethmoreland.github.io/tree/main/files/Code/Hilb3P3_Code}
\end{remark}

The paper is structured as follows: in Section \ref{Section:backgroundP2} we give some background on Hilbert schemes of points in $\PP^2$, their Chow groups, and their nef and effective cones. In Section \ref{section:backgroundP3}, we give some background on $\Hilb_3 \PP^3$ and the structure of its Chow ring. In particular, we recount a basis for the Chow groups given by \cite{RL90}, and extend a basis of \cite{MS90} to the case of $A_i(\Hilb_3 \PP^3)$ for $i=2,3,6,7$. In Section \ref{section:orbits}, we study the orbits of the $\textup{PGL}$ action and their positive cones. In Section \ref{section:nef-eff-Hilb3P3}, we prove the main theorems. Lastly, in Appendix section \ref{appendix:computations} we compute some classes necessary for the main body of the paper, and in Appendix section \ref{appendix:Grassmannian} we examine classes coming from the Grassmannian $\mathbb{G}(1,3)$ and use them to compute a few intersection products.

\subsection*{Acknowledgements}
The author would like to thank Izzet Coskun, Joe Harris, Tim Ryan, Ra\'ul Ch\'avez Sarmiento, and Geoffrey Smith for helpful discussions regarding the content of this paper. The author would also like to thank Serina Hu for assistance with some of the Python code relating to this paper. This work was partially supported by the NSF under grants No. DGE1745303 and No. DMS2037569.
\section{Preliminaries on Hilbert schemes of points}\label{Section:backgroundP2}
\subsection{Groups of cycles and positive cones}
In this paper, our (co)homology, Chow, and numerical groups have $\Q$-cofficients. 
\begin{proposition}\label{prop:chow-hom} For $\Hilb_3 \PP^n$, the cycle map from the Chow groups to singular homology groups
\begin{align*}
\textup{cl}: A_{\bullet}(\Hilb_3 \PP^n) \to H_{\bullet}(\Hilb_3 \PP^n)
\end{align*}
is an isomorphism. In particular, numerical equivalence and rational equivalence align for $\Hilb_3 \PP^n$, and a cycle in $A_k(\Hilb_3 \PP^n)$ is determined by its intersection numbers with a basis of $A_{3n-k}(\Hilb_3 \PP^n)$. 
\end{proposition}
\begin{proof}The same argument in \cite{ES87} Lemma 2.1 applies. The torus action on $\PP^n$ yields a torus action on $\Hilb_3 \PP^n$ with finitely many fixed points (length $3$ schemes corresponding to monomial ideals), and a sufficiently general one parameter subgroup will yield a $\mathbb{G}_m$ action with finitely many fixed points as well. Then Theorem 4.5 in \cite{BB73} yields a cellular decomposition of $\Hilb_3 \PP^n$, which, by Example 19.1.11, \cite{F98}, implies the proposition.
\end{proof}

\begin{definition}
Let $N_k(X)$ denote the $\Q$-vector space of $k$-dimensional algebraic cycles of a variety $X$ modulo numerical equivalence. The effective cone $\Eff_k(X)$ is the cone spanned by classes of $k$-dimensional varieties, and the pseudoeffective cone is its closure $\overline{\Eff}_k(X)$. For the varieties we consider, the pseudoeffective cone and effective cone will align.

Furthermore, let $N^k(X)$ denote codimension $k$ cycles modulo numerical equivalence. We define $\Nef^k(X)$ as the dual cone of $\overline{\Eff}_k(X)$ in $N^k(X)$. 
\end{definition}

For any of the varieties for which we need to compute the effective and nef cones, $N_k(X) \cong A_k(X)$ and $N^k(X) \cong A^k(X)$, so we may view our cones as living in the Chow groups. For $\Hilb_3 \PP^3$, these isomorphisms follow from Proposition $\ref{prop:chow-hom}$, and for the locus of collinear schemes, the locus of collinear schemes support at at most two points, the locus of collinear schemes supported at a point, and the locus of planar fat points, which all appear in Section \ref{section:orbits}, it follows from these varieties being isomorphic to flag varieties or projective bundles over flag varieties.
\subsection{Two bases for the Chow groups of $\Hilb_3 \PP^2$}

We consider two bases for the Chow groups $A_k(\Hilb_3 \PP^2)$: one from Elencwajg \& Le Barz (\cite{EL83}), and one from Mallavibarrena \& Sols (\cite{MS90}).
 The Elencwajg-Le Barz basis has the advantage of being stated in terms of multiplication of just a few classes, and the added perk of having a $\Hilb_3 \PP^3$ analog with much of the multiplication data worked out. The Mallavibarrena-Sols basis is a little easier to compute intersection products with by hand, and additionally Ryan and Stathis \cite{RS21} compute the nef cones of $A^{\bullet}(\Hilb_3 \PP^2)$ in codimensions 2 and 3 and effective cones in dimension 2 and 3 using this basis. We utilize some of their computations in the $\Hilb_3 \PP^3$ case. Both bases will be utilized in this paper, and we compute some changes of coordinates between the two bases in Propositions \ref{prop:P2-conversion} and \ref{prop:P3-conversion}.

 We start with the Elencwajg-Le Barz (EL) basis. Let $Q \in \PP^2$ be a fixed point, $L, L' \subseteq \PP^2$ be two fixed lines, and define the following classes.
 \begin{align*}
F &= [\{\Gamma \in \Hilb_3 \PP^2 : \exists d \subset \Gamma \textup{ of length 2 such that $d$ is collinear with $Q$}\}]\\
H &= [\{\Gamma \in \Hilb_3 \PP^2 : \Gamma \cap L \ne \emptyset\}]\\
\ell &= [\{\Gamma \in \Hilb_3 \PP^2 : Q \in \Gamma\}]\\
g &= [\{\Gamma \in \Hilb_3 \PP^2 : \textup{$\Gamma$ is collinear with $Q$}\}]\\
p &= [\{\Gamma \in \Hilb_3 \PP^2 : \exists d \subset \Gamma \textup{ of length 2 such that $d \subseteq L$}\}]\\
g_e &= [\{\Gamma \in \Hilb_3 \PP^2 : \Gamma \subseteq L\}]
 \end{align*}
and lastly define $\betarep$ to be the locus whose generic member $\Gamma$ is a nonreduced scheme given by the union of a length two subscheme supported on $L$ and a point supported on $L'$.

Then bases for the Chow groups of $\Hilb_3 \PP^2$ are given by:
\begin{center}
  \begin{tabular}{ c|l } 
   $k$ & $A_k(\Hilb_3 \PP^2)$ basis   \\ 
   \hline
   6 & [$\Hilb_3 \PP^2$]  \\ 
   5 & $H,F$ \\ 
   4 & $H^2, HF, \ell, p ,g$\\ 
   3 & $H^3, H^2F, H\ell, Hg, g_e, \betarep$ \\ 
   2 & $H^2\ell, H^2g,\ell^2,\ell g, \ell p$ \\ 
   1 & $H \ell p, H \ell^2$ \\ 
   0 & $\ell^3$\\
  \end{tabular}
  \end{center}
and for reference we include the intersection tables for (co)dimension 1 and (co)dimension 2.
\begin{figure}[h]\begin{center}
  \stepcounter{theorem}
  \begin{tabular}{ c|cc } 
     & $H\ell p$ & $H \ell^2$  \\ 
    \hline
    H & 1 & 1\\
    F & 1 & 2 
   \end{tabular}

   \begin{tabular}{ c|ccccc } 
    & $H^2 \ell$ & $H^2 g$ & $\ell^2$ & $\ell g$ & $\ell p$  \\ 
   \hline
  $H^2$ & 3 & 6 & 1 & 1 & 1\\
  $HC$ & 6 & 3 & 2 & 0 & 1\\
  $\ell$ & 1 & 1 & 1 & 0 & 0\\
  $g$ & 1 & -1 & 0 & 0 & 0\\
  $p$ & 1 & 0 & 0 & 0 & 0\\
  \end{tabular}
\end{center}
\caption{Intersection tables for $A_5(\Hilb_3 \PP^2) \times A_1(\Hilb_3 \PP^2)$ and $A_4(\Hilb_3 \PP^2) \times A_2(\Hilb_3 \PP^2)$ using the EL basis.}
\end{figure}

Next we introduce the Mallavibarrena \& Sols (MS) basis. In general, elements of this basis are defined by three sorts of conditons: being contained in a fixed line with a fixed point included, being contained in a fixed line, and being collinear with a fixed point. We will not go into detail about the MS basis in general, and will instead describe the classes that arise in the $\Hilb_3 \PP^2$ case and provide figures showing a general point $\Gamma \in \Hilb_3 \PP^2$ in each class. A more in-depth treatment can be found in \cite{MS90}.

For the remainder of this section, $L, L',L''$ are fixed distinct lines in $\PP^2$. $Q$ is a fixed point on $L$, and $R$ is a fixed point not contained in any of the three lines.

In dimension one and codimension one, the EL and MS bases are the same. In dimension five the basis elements are $H$, the locus of schemes incident to a line, and $F$, the locus of schemes with a length two subscheme that is collinear with a point $Q$. 
In dimension one the generators are $\phi = H \ell^2 $ and $\psi = H\ell p$.  $\phi$ is the locus whose general member contains two fixed points $Q,R$, and a third point on $L'$ (which notably contains neither $Q$ nor $R$). $\psi$ is the locus whose general member is a scheme containing two fixed points $Q,R$, and a third point on $L$. See Figure \ref{fig:co-dim1P2pics} for pictures of a general member of each scheme.
\begin{figure}[h]
  \stepcounter{theorem}
  \[
\begin{tikzpicture}[scale=.70]
  \draw (0,2.5)--(0,0);
  \draw (0,1) circle (.13);
  \draw (.5,1.5) circle (.13);
  \draw (.5,.5) circle (.13);

  \node at (.25,-.4) {$H$};
\end{tikzpicture}
\quad \quad \quad \quad
\begin{tikzpicture}[scale=.70]
  \draw[dashed] (0,2.5)--(0,0);
  \node at (0,2.25) {$\times$};

  \draw (.5,1) circle (.13);
  \draw (0,1.5) circle (.13);
  \draw (0,.5) circle (.13);

  \node at (.25,-.4) {$F$};
\end{tikzpicture}
\quad \quad \quad \quad \quad
\begin{tikzpicture}[scale=.70]
  \draw (.5,2.5)--(.5,0);

  \draw (.5,1) circle (.13);
  \filldraw (0,1.5) circle (.13);
  \filldraw (0,.5) circle (.13);

  \node at (.25,-.4) {$\phi$};
\end{tikzpicture}
\quad \quad \quad \quad \quad
\begin{tikzpicture}[scale=.70]
  \draw (0,2.5)--(0,0);

  \filldraw (.5,1) circle (.13);
  \filldraw (0,1.5) circle (.13);
  \draw (0,.5) circle (.13);

  \node at (.25,-.4) {$\psi$};
\end{tikzpicture}
\]
\caption{Pictures showing a general member of each of $H, F, \phi, \psi$. A solid point denotes a fixed point that must be included in the scheme; an $\times$ denotes a fixed point that is not necessarily in the support. A solid line denotes a fixed line; a dashed line indicates a line that can vary, and is used for our purposes to show a collinearity requirement.}
\label{fig:co-dim1P2pics}
\end{figure}
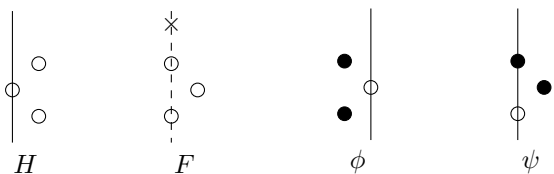

We now turn to dimension 2 and codimension 2. The five basis elements of $A_4(\Hilb_3 \PP^2)$ are as follows. $A$ is the locus of schemes collinear with $Q$. $B$ is the locus whose general member is reduced with a length two subscheme collinear with $Q$, and the remaining point lying on $L$. $C$ is the locus whose general member is a reduced scheme with one point on $L$ and one point on $L'$. $D$ is the locus of schemes with a length two subscheme contained in $L$. $E$ is the locus of schemes containing $Q$.

The five basis elements of $A_2(\Hilb_3 \PP^2)$ are as follows. $\alpha$ is the locus of schemes contained in $L$ and containing $Q$. $\beta$ is the locus whose general member is a reduced scheme containing two points in $L$, with one of the points being $Q$, and with a third point on $L'$. $\gamma$ is the locus whose general member is a reduced scheme containing $Q$, a point on $L'$, and a point on $L''$. $\delta$ is the locus that can be written as a union of $Q$ and a length two scheme contained in $L'$. $\epsilon$ is the locus of schemes containing $Q,R$ in the support. See Figure \ref{fig:co-dim2P2pics} for pictures of a general member of each scheme.

\begin{figure}[h]
  \stepcounter{theorem}
  \[
    \begin{tikzpicture}[scale=.70]
      \draw[dashed] (0,2.5)--(0,0);
      \node at (0,2.25){$\times$};
    
      \draw (0,.5) circle (.13);
      \draw (0,1) circle (.13);
      \draw (0,1.5) circle (.13);
    
      \node at (0,-.4) {$A$};
    \end{tikzpicture}
    \quad \quad \quad \quad \quad
    \begin{tikzpicture}[scale=.70]
      \draw (0,2.5)--(0,0);
      \draw[dashed] (.5,2.5)--(.5,0); 
      \node at (.5,2.25) {$\times$};
    
      \draw (0,1) circle (.13);
      \draw (.5,1.5) circle (.13);
      \draw (.5,.5) circle (.13);
    
      \node at (.25,-.4) {$B$};
    \end{tikzpicture}
    \quad \quad \quad \quad \quad
    \begin{tikzpicture}[scale=.70]
      \draw (0,2.5)--(0,0);
      \draw (.5,2.5)--(.5,0);
    
      \draw (0,1.5) circle (.13);
      \draw (.5,1) circle (.13);
      \draw (1,.5) circle (.13);
    
      \node at (.5,-.4) {$C$};
    \end{tikzpicture}
    \quad \quad \quad \quad \quad
    \begin{tikzpicture}[scale=.70]
      \draw (0,2.5)--(0,0);
    
      \draw (.5,1) circle (.13);
      \draw (0,1.5) circle (.13);
      \draw (0,.5) circle (.13);
    
      \node at (.25,-.4) {$D$};
    \end{tikzpicture}
    \quad \quad \quad \quad \quad
    \begin{tikzpicture}[scale=.70]
    
      \filldraw (0,1) circle (.13);
      \draw (.5,1.5) circle (.13);
      \draw (.5,.5) circle (.13);
    
      \node at (.25,-.4) {$E$};
    \end{tikzpicture}
    \]
    \\
    \[
    \begin{tikzpicture}[scale=.70]
      \draw (0,2.5)--(0,0);
    
      \draw (0,.5) circle (.13);
      \draw (0,1) circle (.13);
      \filldraw (0,1.5) circle (.13);
    
      \node at (0,-.4) {$\alpha$};
    \end{tikzpicture}
    \quad \quad \quad \quad \quad
    \begin{tikzpicture}[scale=.70]
      \draw (0,2.5)--(0,0);
      \draw (.5,2.5)--(.5,0); 
    
      \draw (.5,1) circle (.13);
      \filldraw (0,1.5) circle (.13);
      \draw (0,.5) circle (.13);
    
      \node at (.25,-.4) {$\beta$};
    \end{tikzpicture}
    \quad \quad \quad \quad \quad
    \begin{tikzpicture}[scale=.70]
      \draw (.5,2.5)--(.5,0);
      \draw (1,2.5)--(1,0);
    
      \filldraw (0,1.5) circle (.13);
      \draw (.5,1) circle (.13);
      \draw (1,.5) circle (.13);
    
      \node at (.5,-.4) {$\gamma$};
    \end{tikzpicture}
    \quad \quad \quad \quad \quad
    \begin{tikzpicture}[scale=.70]
      \draw (.5,2.5)--(.5,0);
    
      \filldraw (0,1) circle (.13);
      \draw (.5,1.5) circle (.13);
      \draw (.5,.5) circle (.13);
    
      \node at (.25,-.4) {$\delta$};
    \end{tikzpicture}
    \quad \quad \quad \quad \quad
    \begin{tikzpicture}[scale=.70]
    
      \draw (.5,1) circle (.13);
      \filldraw (0,1.5) circle (.13);
      \filldraw (0,.5) circle (.13);
    
      \node at (.25,-.4) {$\epsilon$};
    \end{tikzpicture}
    \]
\caption{General members of each of the basis elements of $A_4(\Hilb_3 \PP^2), A_2 (\Hilb_3 \PP^2)$.}
\label{fig:co-dim2P2pics}
\end{figure}
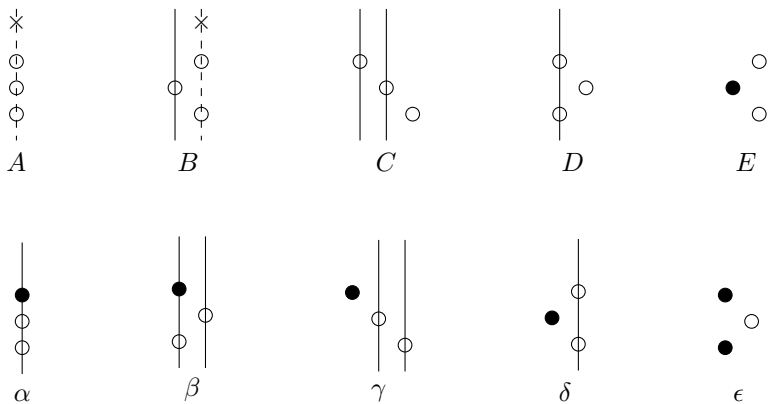
The intersection pairings between these bases are given below.
\begin{table}[h]
  \stepcounter{theorem}
\begin{center}
  \begin{tabular}{ c|ccccc } 
   & $\alpha$ & $\beta$ & $\gamma$ & $\delta$ &$\epsilon$  \\ 
  \hline
    $A$ & 0 & 0 & 1 & 0 & 0 \\
    $B$ & 0 & 1 & 2 & 1 & 0 \\
    $C$ & 1 & 2 & 2 & 1 & 0 \\
    $D$ & 0 & 1 & 1 & 0 & 0\\
    $E$ & 0 & 0 & 0 & 0 & 1\\
 \end{tabular}
\end{center}
\caption{Intersection matrix for $A_4(\Hilb_3 \PP^2) \times A_2(\Hilb_3 \PP^2)$.}
\label{table:co-dim2P3}
\end{table}
We now describe the six basis elements of $A_3(\Hilb_3 \PP^2)$. $U$ is the locus of schemes containing $Q$ and incident to $L'$. $V$ is the locus of schemes whose general member is the union of a fixed point $Q$ and a length two subscheme collinear with $R$. $W$ is the locus of schemes whose general member is a reduced subscheme given by the union of one point each on $L,L',L''$. $X$ is the locus of schemes whose general member is the union of a length two subscheme contained in $L$ and a point contained in $L'$. $Y$ is locus of schemes contained in $L$. $Z$ is the locus of schemes whose general member contains two points on $L$, one of which is $Q$. See Figure \ref{fig:co-dim3P2pics} for a picture describing a general member of each locus, and Table \ref{table:co-dim3P2} for the intersection products. 

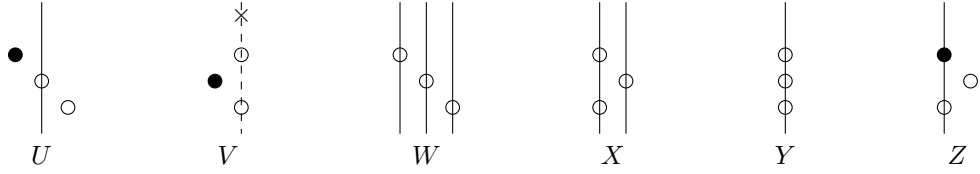
\begin{figure}[h]
  \stepcounter{theorem}
  \[
    \begin{tikzpicture}[scale=.70]
      \draw (.5,2.5)--(.5,0);
    
      \filldraw (0,1.5) circle (.13);
      \draw (.5,1) circle (.13);
      \draw (1,.5) circle (.13);
    
      \node at (.5,-.4) {$U$};
    \end{tikzpicture}
    \quad \quad \quad \quad \quad
    \begin{tikzpicture}[scale=.70]
      \draw[dashed] (.5,2.5)--(.5,0);
      \node at (.5,2.25){$\times$};
    
      \filldraw (0,1) circle (.13);
      \draw (.5,1.5) circle (.13);
      \draw (.5,.5) circle (.13);
    
      \node at (.25,-.4) {$V$};
    \end{tikzpicture}
    \quad \quad \quad \quad \quad
    \begin{tikzpicture}[scale=.70]
      \draw (0,2.5)--(0,0);
      \draw (.5,2.5)--(.5,0);
      \draw (1,2.5)--(1,0);
    
      \draw (0,1.5) circle (.13);
      \draw (.5,1) circle (.13);
      \draw (1,.5) circle (.13);
    
      \node at (.5,-.4) {$W$};
    \end{tikzpicture}
    \quad \quad \quad \quad \quad
    \begin{tikzpicture}[scale=.70]
      \draw (0,2.5)--(0,0);
      \draw (.5,2.5)--(.5,0);
    
      \draw (.5,1) circle (.13);
      \draw (0,1.5) circle (.13);
      \draw (0,.5) circle (.13);
    
      \node at (.25,-.4) {$X$};
    \end{tikzpicture}
    \quad \quad \quad \quad \quad
    \begin{tikzpicture}[scale=.70]
      \draw (0,2.5)--(0,0);
    
      \draw (0,.5) circle (.13);
      \draw (0,1) circle (.13);
      \draw (0,1.5) circle (.13);
    
      \node at (0,-.4) {$Y$};
    \end{tikzpicture}
    \quad \quad \quad \quad \quad
    \begin{tikzpicture}[scale=.70]
      \draw (0,2.5)--(0,0);
    
      \draw (.5,1) circle (.13);
      \filldraw (0,1.5) circle (.13);
      \draw (0,.5) circle (.13);
    
      \node at (.25,-.4) {$Z$};
    \end{tikzpicture}
    \]
\caption{General members of the basis elements of $A_3(\Hilb_3 \PP^2)$}
\label{fig:co-dim3P2pics}
\end{figure}
\begin{table}[h]
  \stepcounter{theorem}
\begin{center}
  \begin{tabular}{ c|cccccc } 
   & $U$ & $V$ & $W$ & $X$ &$Y$ & $Z$ \\ 
  \hline
 $U$ & 1 & 1 & 0 & 0 & 0 & 1\\
 $V$ & 1 & 1 & 0 & 0 & 0 & 0\\
 $W$ & 0 & 0 & 6 & 3 & 1 & 0\\
 $X$ & 0 & 0 & 3 & 1 & 0 & 0\\
 $Y$ & 0 & 0 & 1 & 0 & 0 & 0\\
 $Z$ & 1 & 0 & 0 & 0 & 0 & 1
 \end{tabular}
\end{center}
\caption{Intersection table for $A_3(\Hilb_3 \PP^2) \times A_3(\Hilb_3 \PP^2)$}
\label{table:co-dim3P2}
\end{table}
\begin{proposition}\label{prop:P2-conversion} We have the following change of basis between the EL and MS bases in $A_2(\Hilb_3 \PP^2)$:
\begin{align*}
H^2\ell &=  \gamma + \varepsilon \\
H^2 g &= 3 \alpha + \beta - \gamma + 2\delta + \epsilon\\ 
\ell^2 &= \epsilon\\
\ell g &= \alpha\\
\ell p &= \delta
\end{align*}
Correspondingly, we have the following conversion between the EL and MS bases in $A_4(\Hilb_3 \PP^2)$.
\begin{align*}
H^2 &= C+E\\
HC &= B+2D+2E\\
\ell &= E\\
g &= A\\
p &= D
\end{align*}
\end{proposition}
\begin{proof}
The dimension 4 change of basis can be obtained by the dimension 2 change of basis. In dimension 2, the only nontrivial computation is $H^2g$. We compute its expression in the MS basis by first computing the class of $Hg$. We see that $Hg$ has intersection numbers:
\begin{center}
  \begin{tabular}{ c|cccccc } 
  & $U$ & $V$ & $W$ & $X$ & $Y$ & $Z$ \\
  \hline
  $Hg$ & 1 & 0 & 3 & 0 & 0 & 0
 \end{tabular}
\end{center}
Hence $Hg = (Z-U+V)+3Y$. Then $Z \cdot H = \beta + \epsilon$, $U \cdot H = \gamma + \epsilon$, and $Y \cdot H = \alpha$. $V \cdot H$ can be checked by computing the intersection products with the complementary dimension:
\begin{center}
  \begin{tabular}{ c|ccccc } 
  & $A$ & $B$ & $C$ & $D$ & $E$ \\
  \hline
  $V\cdot H$ & 0 & 2 & 2 & 0 & 1 
 \end{tabular}
\end{center}
yielding that $V \cdot H = 2 \delta + \epsilon$. This yields $H^2 g = 3 \alpha + \beta - \gamma + 2\delta + \epsilon$.
\end{proof}
\subsection{Techniques for computing the effective and nef cones}\label{subsec:techniques-for-eff-nef}

We give an overview of the methods in Ryan and Stathis's work \cite{RS21} computing some of the nef and effective cones of $\Hilb_3 \PP^2$.


In Theorem 5 and 6 of \cite{RS21}, the authors prove the following result. 
\begin{theorem}[Ryan, Stathis '21]\label{theorem:(co)dim2P2} We have:
\begin{align*}
  \Eff_2(\Hilb_3 \PP^2) &= \langle \alpha, \epsilon, -\alpha + \delta, -\alpha + \beta - \delta, \alpha - \beta + \gamma - \delta, -2\alpha + \beta -\gamma + 2\delta + \epsilon \rangle \\
  \Nef^2(\Hilb_3 \PP^2) &= \langle B,C,D,E,A+B,A+E \rangle
\end{align*}
And furthermore:
\begin{align*}
  \Eff_3(\Hilb_3 \PP^2) &= \langle Y, -U+V+Z, 3U-2V-W+4X-6Y-Z,\\
  & \quad \quad  W-3X+3Y,X-3Y,U-V,U-Z\rangle\\
  \Nef^3(\Hilb_3 \PP^2) &= \langle U,V,W,X,Z,V+Y,X+Y,2Y+Z\rangle.
\end{align*}
\end{theorem}

In a variety with a transitive action by an algebraic group $G$,   Kleiman's transversality theorem is a useful tool for determining nef and effective cones. For example, Kleiman's allows one to determine that the nef and effective cones of a Grassmannian are spanned by the Schubert classes. While the action of $\textup{PGL}_4(\mathbb{C})$ on $\Hilb_3 \PP^3$ is not transitive, it does have finitely many orbits, which is enough to yield the following useful tool.

\begin{lemma}\label{lemma:kleiman-analog}Let $G$ be an infinite connected algebraic group acting on a smooth variety $X$ with finitely many orbits $\OO_1,\dots,\OO_m$. Suppose $Y$ is a subvariety of $X$ of codimension $k$. If $Y$ intersects $\OO_1,\dots,\OO_{\ell}$ in the expected dimension, intersects $\OO_{\ell+1},\dots, \OO_{m}$ in dimension one higher than expected, and intersects every effective $k$-cycle in $\overline{\OO_{\ell+1}},\dots,\overline{\OO_{m}}$ non-negatively, then $Y$ is nef.
\end{lemma}
\begin{proof}
See Lemma 3 in \cite{RS21}.
\end{proof}

In Ryan and Stathis's work, understanding the nef and effective cones of the closures of the orbits of the $\textup{PGL}_3(\C)$ action on $\Hilb_3 \PP^2$ is key to working out $\Nef_{4}, \Nef_3, \Eff_2, \Eff_3$ of $\Hilb_3 \PP^2$. Thus we too work out many effective and nef cones of the orbit closures in Section \ref{section:orbits}. In particular, this lemma suggests that we should compute the effective cones of the orbits, take the cone $\mathcal{C}$ generated by all the generators of the aforementioned cones, and use Lemma \ref{lemma:kleiman-analog} to show that the dual cone $\mathcal{C}^{\vee}$ consists of nef classes, which would prove we have determined the nef and effective cones. 

We also wish to identify these ``positive'' cones computed in $\Hilb_3 \PP^2$ as part of the nef and effective cones in $\Hilb_3 \PP^3$. To do this we will need a notion of ``lifting'' cycles in $\Hilb_3 \PP^2$, which we will define in Section \ref{section:backgroundP3}. Lifting cycles will also help us produce a good basis for applying Lemma \ref{lemma:kleiman-analog}, since the lemma only applies to effective classes.
\section{The Chow ring of $\Hilb_3 \PP^3$}\label{section:backgroundP3}
\subsection{Basic structure} $\Hilb_3 \PP^3$ is an irreducible, projective, smooth scheme and its Chow ring is isomorphic to the cohomology ring \cite{EL83}. Any $\Gamma \in \Hilb_3 \PP^3$ is contained in a plane in $\PP^3$, and if $\Gamma$ is not contained in a line that plane is unique. This motivates considering the incidence variety $\wt{\Hilb}_3\PP^3$.
\[\wt{\Hilb}_3\PP^n = \{(\Gamma, \Pi) \in \Hilb_3 \PP^3 \times \mathbb{G}(2,n) : \Gamma \subseteq \Pi\}.\]
Note that $\wt{\Hilb}_3\PP^n$ is a $\Hilb_3 \PP^2$-bundle over $\mathbb{G}(2,n)$. Rossell\'{o}-Llompart uses this presentation to give an analogue of the EL basis for $\Hilb_3 \PP^3$ in \cite{RL90}. We give an abbreviated primer on this basis for the reader. Let $\Al_3\PP^n$ denote the locus of collinear schemes in $\Hilb_3 \PP^n$. 
\begin{proposition}
Let $q: \wt{\Hilb}_3 \PP^n \to \Hilb_3 \PP^n$ be the projection map onto the first coordinate. Then $q$ is the blowup of $\Hilb_3 \PP^n$ along $\Al_3 \PP^n$.
\end{proposition}
\begin{proof}
  See \cite{EL83}.
\end{proof}
Consider the following commutative diagram, using $\wt{\Al}_3\PP^n$ to denote the proper transform of $\Al_3 \PP^n$:
\[\begin{tikzcd}
  \wt{\Al}_3 \PP^n \arrow[r,"j"] \arrow[d,"q'"]& \wt{\Hilb}_3 \PP^n \arrow[d,"q"]\\
  \Al_3 \PP^n \arrow[r,"h"]& \Hilb_3 \PP^n
\end{tikzcd}\]
Roberts \cite{R88} as well as Rossell\'{o}-Llompart and Xamb\'{o}-Descamps \cite{RLXD92} use this to generate the split exact sequences
\[0 \to \left(\Ker {q'}_*\right)_k \stackrel{j_*}{\to} A_k(\wt{\Hilb}_3 \PP^n) \stackrel{q_*}{\to} A_k(\Hilb_3 \PP^n) \to 0\]
for $k \ge 0$. Then one can use bases of $\left(\Ker {q'}_*\right)_k$, $A_k(\wt{\Hilb}_3 \PP^n)$ to write a basis of $A_k(\Hilb_3 \PP^n)$; see Section 3 in \cite{RL90} for more details. 

To define the basis of $\Hilb_3 \PP^3$ given in \cite{RL90}, we first need to define some loci. Let $\Pi,\Pi'$ denote fixed planes, $L$ denote a fixed line, and $Q$ denote a fixed point.
\begin{align*}
H &= \{\Gamma \in \Hilb_3 \PP^3 : \Gamma \cap \Pi \ne \emptyset\}\\
P &= \{\Gamma \in \Hilb_3 \PP^3 : \Gamma \textup{ is coplanar with }Q\}\\
P_3 &= \{\Gamma \in \Hilb_3 \PP^3: \Gamma \subseteq \Pi\}\\
P_2 &= \{\Gamma \in \Hilb_3 \PP^3 : \Gamma \textup{ coplanar with }L\}\\
\ell &=\{\Gamma \in \Hilb_3 \PP^3 : \Gamma \cap L \ne \emptyset\}\\
p &= \{\Gamma \in \Hilb_3 \PP^3 : \exists d \subseteq \Gamma \text{ of length 2 such that } d \subseteq \Pi\}
\end{align*}
In addition, we define $\betarep$ to be the locus of schemes in $\Hilb_3 \PP^3$ whose general member is the union of a length two scheme contained in $\Pi$ and a point contained in $\Pi'$. It will also be useful to define
\[F = \{\Gamma \in \Hilb_3 \PP^3 : \exists d \subseteq \Gamma \text{ of length 2 such that } d \textup{ coplanar with $L$}\}.\]
\begin{lemma}\label{lem:F=P+H}The following identity holds in $A_8(\Hilb_3 \PP^3)$.
\begin{equation}
F = H + P
\end{equation}
\end{lemma}
\begin{proof} This follows from intersecting each side of the equation with $P_3 H \ell p, P_3 H \ell^2$, which constitute a basis of $A_1(\Hilb_3 \PP^3)$. 
\end{proof}
\begin{remark} By abuse of terminology, $H, \ell, p, \betarep, F$ are being used in two ways (for both a subvariety of $\Hilb_3 \PP^2$ and a subvariety of $\Hilb_3 \PP^3$). We do this for convenience, as it is clear that $H, \ell, p, \betarep, F$ restrict to their $\Hilb_3 \PP^2$ analogues when we consider $\Hilb_3 \PP^2 \subseteq \Hilb_3 \PP^3$ (by identifying $\PP^2 \hookrightarrow \PP^3$ as a plane in $\PP^3$), and because in context it will be clear which usage we are referring to in any given situation.
\end{remark}
\begin{lemma} The Chow groups of $\Hilb_3 \PP^3$ have the following bases, which henceforth we refer to as the $\RL$ bases.
\[
  \begin{tabular}{c|l}
  $k$ & Basis for $A_k(\Hilb_3 \PP^3)$\\
  \hline
  9 & $\Hilb_3 \PP^3$\\
  8 & $P,H$\\
  7 & $P_2, P^2, PH, H^2, \ell, p$\\
  6& $P_3,P_2 P,P_2 H, PH^2, P^2H, P\ell, Pp, H^3, H\ell, \betarep$\\
  5 & $P_3 P, P_3 H, P_2 H^2, P_2 HP, P_2 \ell, P_2^2, P_2 p$, $(P^2 H^2 + PH^3)$, $PH \ell, P\betarep, H^2 \ell, \ell^2, \ell p$ \\
  4 & $P_3 HP, P_3 H^2, P_3 \ell, P_3 P_2, P_3 p, P_2 H \ell, P_2^2 H, P_2 \betarep, P_2 H^3, P_2 P H^2, PH^2 \ell, P \ell^2, H \ell p$ \\
  3 & $P_3 H \ell, P_3 P_2 H, P_3^2, P_3 \betarep, P_3 H^3, P_3PH^2$, $(P_2^2H^2 + P_2 H^2 \ell)$, $P_2 \ell p, P_2 \ell^2, P H \ell p$ \\
  2 & $P_3 H^2 \ell, P_3 P_2 H^2, P_3 \ell^2, P_3 \ell p, P_3 P_2 \ell, P_2 H \ell p$\\
  1 & $P_3 H \ell p, P_3 H \ell^2$ \\
  0 & $P_3 \ell^3$
  \end{tabular}
\]
\begin{proof}
See Lemma 4.5 in \cite{RL90}.
\end{proof}
\end{lemma}

\begin{remark}Rossell\'{o}-Llompart provides the integer values of all weight $9$ monomials in $P,H,P_3,\betarep,P_2 \ell,p$ in the tables of Proposition 4.7 of \cite{RL90}. Hence if you can write a cycle in $\Hilb_3 \PP^3$ in terms of $P,H,P_3,\betarep,P_2 \ell,p$, then you can use this table and the intersection tables below to write the cycle in terms of the $\RL$ basis. Hence in this paper we will aim to write all pertinent cycles in terms of these seven generators, but may not always give their expression in the bases of $A_k(\Hilb_3 \PP^3)$.
\end{remark}
For reference, we include some of the intersection tables for this basis in Tables \ref{table:(co)dim2-P3}, \ref{table:(co)dim3-P3}, \ref{table:(co)dim4-P3}.
\begin{table}[h]
  \stepcounter{theorem}
  \[
\begin{tabular}{c|cccccc}
 & $P_2$ & $P^2$ & $PH$ & $H^2$ & $\ell$ & $p$\\
 \hline
$P_3 H^2 \ell$ & 1 & 0 & 3 & 3 & 1 & 1\\
$P_3 P_2 H^2$ & -1 & 0 & -3 & 6 & 1 & 0\\
$P_3 \ell^2$ & 0 & -1 & 1 & 1 & 1 & 0\\
$P_3 \ell p$ & 0 & 0 & 0 & 1 & 0 & 0\\
$P_3 P_2 \ell$ & 0 & 1 & -1 & 1 & 0 & 0\\
$P_2 H \ell p$ & 0 & -3 & 4 & 9 & 2 & 2
\end{tabular}\]
\caption{The intersection matrix for $A_2(\Hilb_3 \PP^3) \times A_7(\Hilb_3 \PP^3)$}
\label{table:(co)dim2-P3}
\end{table}
\begin{table}[h]
  \stepcounter{theorem}
  \begin{center}
  \begin{tabular}{c|cccccccccc}
   & $P_3$ & $P_2 P$ & $P_2 H$ & $PH^2$ & $P^2 H$ & $P \ell$ & $P p$ & $H^3$ & $H \ell$ & $\betarep$\\
   \hline
  $P_3 H \ell$ & 0 & -1 & 1 & 3 & 0 & 1 & 0 & 3 & 1 & 0\\
  $P_3 P_2 H $ & 0 & 1 & -1 & -3 & 0 & -1 & 0 & 6 & 1 & 0\\
  $P_3^2$ & 0 & 0 & 0 & -1 & 1 & 0 & 0 & 1 & 0 & 0\\
  $P_3 \betarep$ & 0 & 0 & 0 & 1 & -1 & 0 & 0 & 3 & 0 & 1\\
  $P_3 H^3$ & 1 & -3 & 6 & 15 & 3 & 3 & 3 & 15 & 3& 3\\
  $P_3 P H^2$ & -1 & 0 & -3 & 3 & -12 & 0 & -3 & 15 & 3 & 1\\
  $(P_2^2 H^2 + P_2 H^2 \ell)$ & 0 & -2 & 2 & 21& -6&4&1&65&16&4\\
  $P_2 \ell p$ & 0 & 0 & 0 & 4 & -3 & 0 & 0 & 9 & 2 & 1\\
  $P_2 \ell^2$ & 0 & -2 & 2 & 7 & 1 & 3 & 0 & 7 & 3 & 0\\
  $P H \ell p$ & 0 & -3 & 4 & 18& 4 & 5 & 2 & 25 & 7 & 2
  \end{tabular}
  \end{center}
  \caption{The intersection matrix for $A_3(\Hilb_3 \PP^3) \times A_6(\Hilb_3 \PP^3)$}
  \label{table:(co)dim3-P3}
  \end{table}
\begin{table}[h]\hspace{-.625cm}
  \stepcounter{theorem}
  \setlength{\tabcolsep}{5pt}
  \begin{tabular}{c|ccccccccccccc}
    &$P_3P$ & $P_3 H$ & $P_2 H^2$ & $P_2 HP$ & $P_2 \ell$ & $P_2^2$ & $P_2 p$ & $(P^2 H^2 + PH^3)$ & $PH \ell$ & $P \betarep$ & $H^2 \ell$ & $\ell^2$ & $\ell p$\\
    \hline
  $P_3 HP$ &1&-1&-3&0&-1&1&0&-9&0&-1&3&1&0\\
  $P_3 H^2$ &-1&1&6&-3&1&-1&0&18&3&1&3&1&1\\
  $P_3 \ell$ &0&0&1&-1&0&0&0&3&1&0&1&1&0\\
  $P_3 P_2$ &0&0&-1&1&0&0&0&-3&-1&0&1&0&0\\
  $P_3 p$ &0&0&0&0&0&0&0&0&0&0&1&0&0\\
  $P_2 H \ell$ &-1&1&9&-3&2&-2&0&33&7&1&7&3&2\\
  $P_2^2 H$ &1&-1&-7&1&-2&2&0&-18&-3&-1&9&2&0\\
  $P_2 \betarep$ &0&0&1&-1&0&0&0&7&1&1&3&0&1\\
  $P_2 H^3$ &-3&6&40&-4&9&-7&3&150&25&9&25&7&9\\
  $P_2 P H^2$ &0&-3&-4&-14&-3&1&-3&6&8&-2&25&7&4\\
  $PH^2 \ell$ &0&3&25&8&7&-3&4&116&20&6&20&6&7\\
  $P \ell^2$&-1&1&7&1&3&-2&0&32&6&1&6&0&3\\
  $H \ell p$&0&1&9&4&2&0&2&43&7&2&7&3&2\\
  \end{tabular}
  \caption{The intersection matrix for $A_4(\Hilb_3 \PP^3) \times A_5(\Hilb_3 \PP^3)$}
  \label{table:(co)dim4-P3}
\end{table}

\subsection{Lifts and pushforwards of the MS basis and extending in (co)dimension 2} \label{Subsec:lifts-MS} In this subsection and the next, we produce an analogue of the MS basis for $A_i(\Hilb_3 \PP^3)$ for $i=2,3,6,7$. The motivation for this comes from a few factors, the first of which is that the MS basis for $\Hilb_3 \PP^2$ had properties that made it particularly well-suited to computing the nef and effective cones. The generators $A,B,C,D,E$ of $A_4(\Hilb_3 \PP^2)$ and $U,V,W,X,Y,Z$ of $A_3(\Hilb_3 \PP^2)$ intersect the PGL-orbits in at most one dimension more than expected, and have the property that their duals in $A_2(\Hilb_3 \PP^2), A_3(\Hilb_3 \PP^2)$, are effective. That is, if one follows the strategy outlined in Section \ref{subsec:techniques-for-eff-nef}, you get a candidate nef cone whose rays are non-negative coefficient sums of MS basis elements, which means we can apply Lemma \ref{lemma:kleiman-analog}.

Additionally, extending the MS basis helps to better compare our cones with the $\Hilb_3 \PP^2$ results and allows us to utilize some of the computations from the $\PP^2$ case, as we will see in Sections \ref{section:orbits} and Appendix \ref{appendix:computations}. Furthermore, this leads to cleaner expressions for the final cones than in the basis of Rossell\'{o}-Llompart. Since the dimension of $A_2(\Hilb_3 \PP^k)$ stabilizes for $k\ge 4$, one hopes that eventually we can find a clean expression for $\Eff_2(\Hilb_3 \PP^k), \Nef^2(\Hilb_3 \PP^k)$ for all $k$.

Our first goal is to identify subspaces of the Chow groups of $\Hilb_3 \PP^3$ that "act" like a copy of the MS basis in $\Hilb_3 \PP^2$. That is, we want to identify the intersection matrix of $A_k(\Hilb_3 \PP^2), A^k(\Hilb_3 \PP^2)$ as a submatrix of $A_k(\Hilb_3 \PP^3), A^k(\Hilb_3 \PP^3$), and we would also like the elements of $A(\Hilb_3 \PP^3)$ yielding this submatrix to be defined based on MS basis elements.

Embed $\PP^2 \stackrel{i}{\hookrightarrow} \PP^3$ by identifying it with a distinguished plane $\Pi$ in $\PP^3$. Then certainly we have
\begin{align*}
  i: A_{\bullet}(\Hilb_3 \PP^2) &\to A_{\bullet}(\Hilb_3 \PP^3)\\
  [X] &\mapsto [X]
\end{align*}
where on the right, $X$ is viewed as a subvariety of $\Hilb_3 \PP^3$. For a class in the MS basis, we will generally use previous notation for classes (like $\alpha$ or $U$, and so on) to denote both the class in $\Hilb_3 \PP^2$ and in $\Hilb_3 \PP^3$, instead of writing $i(\alpha)$ or $i(U)$. It will be clear from context which usage we are referring to.

For a map $A^{\bullet}(\Hilb_3 \PP^2) \to A^{\bullet}(\Hilb_3\PP^3)$, we define the notion of a lift of an MS basis element. For each basis element $\mathcal{B}$ of $A(\Hilb_3 \PP^2)$, we want an element $\wt{\mathcal{B}}$ such that $\wt{\mathcal{B}} \cdot P_3  = \mathcal{B}$. We accomplish this by changing all the conditions to their higher-dimensional analogues: the condition of being incident to a line in $\PP^2$ becomes the condition of being incident to a plane in $\PP^3$; the condition of containing a fixed point becomes being incident to a fixed line in $\PP^3$; the condition of being collinear with a point in $\PP^2$ becomes being coplanar with a line in $\PP^3$. 

For example, $\wt{D}$ is the locus of schemes containing a length two subscheme contained in a fixed plane $\Pi$, and $\wt{E}$ is the locus of schemes incident to a line. $\wt{B}$ is the locus whose general member is two points coplanar with a fixed line, and a third point on a fixed plane. Any lifts preserve the ``nesting'' of incidence conditions that we see with classes like $\alpha$: for example, $\wt{\alpha}$ is the locus of schemes contained in a fixed plane, with one point on a fixed line in that fixed plane.

See Figures \ref{bigfig1}, \ref{bigfig2}, \ref{bigfig3}, and \ref{bigfig4} for illustrations of a generic member of each lift.

\begin{figure}[h]
  \stepcounter{theorem}
  \begin{center}
  \begin{tikzpicture}[scale=.70]
    \filldraw[fill=white](0,1)--(1,2)--(3,1)--(2,0)--(0,1);
  
    \draw (1,1.25) circle (.13);
    \draw (1,0) circle (.13);
    \draw (1.5,-.25) circle (.13);
    \node at (1.5,-1) {$\wt{H}$};
  \end{tikzpicture} \quad \quad \quad \quad
  \begin{tikzpicture}[scale=.70]
    \filldraw[fill=white,dashed](0,1)--(1,2)--(3,1)--(2,0)--(0,1);
    \draw[thick] (1,2)--(3,1);
  
  \draw (1,1.25) circle (.13);
  \draw (1.5,1) circle (.13);
  \draw (1,0) circle (.13);
  \node at (1.5,-1) {$\wt{F}$};
  \end{tikzpicture}
  \quad \quad \quad \quad
  \begin{tikzpicture}[scale=.70]
    \filldraw[fill=white](0,1)--(1,2)--(3,1)--(2,0)--(0,1);
    \draw (1,1.25) circle (.13);
    
    \draw (-.5,1.75)--(-.5,-.5);
    \draw (-1,1.75)--(-1,-.5);
    \draw (-.5,1) circle (.13);
    \draw (-1,.5) circle (.13);
    \node at (1,-1) {$\wt{\phi}$};
  \end{tikzpicture}\quad \quad \quad \quad
  \begin{tikzpicture}[scale=.70]
    \filldraw[fill=white](0,1)--(1,2)--(3,1)--(2,0)--(0,1);
    \draw (1,1.25) circle (.13);
    \draw (1.5,1) circle (.13);
    \draw (.5,.75)--(1.5,1.75);
  
    \draw (-.5,1.75)--(-.5,-.5);
    \draw (-.5,1) circle (.13);
    \node at (1.25,-1) {$\wt{\psi}$};
  \end{tikzpicture}\end{center}
  \caption{A picture of a general member of each of the lifts of the MS basis elements of codimension $1$. Solid lines denote fixed lines/planes, and dashed lines denote planes that can vary. In the picture for $\wt{F}$, the locus of schemes with a length two subscheme coplanar with a line $L$, the solid line in the dashed plane denotes that the plane containing the two points can vary, but that the plane must always contain the solid line.}
  \label{bigfig1}
  \end{figure}
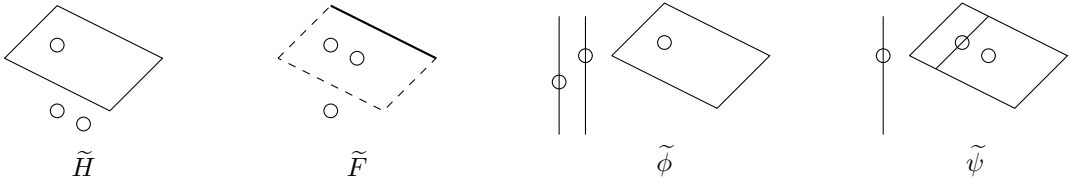
  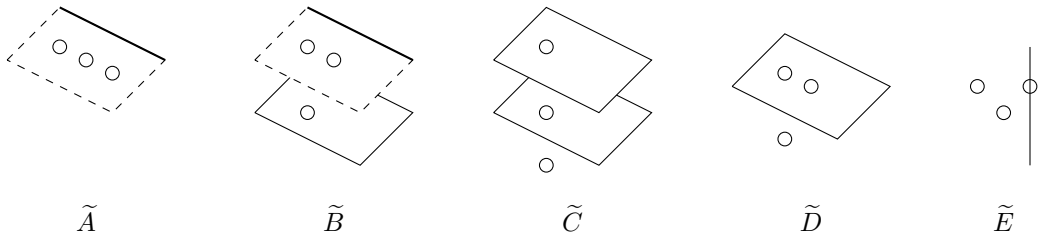
\begin{figure}[h]
    \stepcounter{theorem}
    \begin{center}
    \begin{tikzpicture}[scale=.70]
      \filldraw[fill=white,dashed](0,1)--(1,2)--(3,1)--(2,0)--(0,1);
      \draw[thick] (1,2)--(3,1);

      \draw (1,1.25) circle (.13);
      \draw (1.5,1) circle (.13);
      \draw (2,.75) circle (.13);
      
      \node at (1.5,-2) {$\wt{A}$};
      \end{tikzpicture} \quad \quad \quad
      \begin{tikzpicture}[scale=.70]
        \filldraw[fill=white](0,0)--(1,1)--(3,0)--(2,-1)--(0,0);
        \filldraw[fill=white,dashed](0,1)--(1,2)--(3,1)--(2,0)--(0,1);
        \draw[thick] (1,2)--(3,1);
        \draw (1,1.25) circle (.13);
        \draw (1.5,1) circle (.13);
        \draw (1,0) circle (.13);
  
        \node at (1.5,-2) {$\wt{B}$};
      \end{tikzpicture}\quad \quad \quad  
      \begin{tikzpicture}[scale=.70]
          \filldraw[fill=white](0,0)--(1,1)--(3,0)--(2,-1)--(0,0);
          \filldraw[fill=white](0,1)--(1,2)--(3,1)--(2,0)--(0,1);
          \draw (1,1.25) circle (.13);
          \draw (1,0) circle (.13);
          \draw (1,-1) circle (.13);
          \node at (1.5,-2) {$\wt{C}$};
        \end{tikzpicture}\quad\quad\quad
        \begin{tikzpicture}[scale=.70]
          \filldraw[fill=white](0,1)--(1,2)--(3,1)--(2,0)--(0,1);
          \draw (1,1.25) circle (.13);
          \draw (1.5,1) circle (.13);
          \draw (1,0) circle (.13);
          \node at (1.5,-1.5) {$\wt{D}$};
        \end{tikzpicture}\quad\quad\quad
        \begin{tikzpicture}[scale=.70]
          \draw (-.5,1.75)--(-.5,-.5);

          \draw (-.5,1) circle (.13);
          \draw (-1,.5) circle (.13);
          \draw (-1.5,1) circle (.13); 
          \node at (-1,-1.5) {$\wt{E}$};
        \end{tikzpicture}

    \caption{A picture of a general member of each of the lifts of the MS basis elements of codimension $2$.}
    \label{bigfig2}\end{center}
  \end{figure}
  \begin{figure}[h]
    \stepcounter{theorem}
    \begin{center}
    \begin{tikzpicture}[scale=.70]
      \filldraw[fill=white](0,1)--(1,2)--(3,1)--(2,0)--(0,1);
      \draw (1,1.25) circle (.13);
      \draw (1.5,1) circle (.13);
      \draw (2,.75) circle (.13);
      \draw (.5,.75)--(1.5,1.75);
      \node at (1.5,-2) {$\wt{\alpha}$};
    \end{tikzpicture}\quad\quad\quad 
    \begin{tikzpicture}[scale=.70]
      \filldraw[fill=white](0,0)--(1,1)--(3,0)--(2,-1)--(0,0);
      \filldraw[fill=white](0,1)--(1,2)--(3,1)--(2,0)--(0,1);
      \draw (1,1.25) circle (.13);
      \draw (1.5,1) circle (.13);
      \draw (1,0) circle (.13);
      \draw (.5,.75)--(1.5,1.75);
      \node at (1.5,-2) {$\wt{\beta}$};
    \end{tikzpicture}\quad\quad\quad
    \begin{tikzpicture}[scale=.70]
      \filldraw[fill=white](0,0)--(1,1)--(3,0)--(2,-1)--(0,0);
      \filldraw[fill=white](0,1)--(1,2)--(3,1)--(2,0)--(0,1);
      \draw (1,1.25) circle (.13);
      \draw (1,0) circle (.13);
      
      \draw (-.5,1.75)--(-.5,-.5);
      \draw (-.5,1) circle (.13);
      
      \node at (1,-2) {$\wt{\gamma}$};
    \end{tikzpicture}\quad\quad\quad
    \begin{tikzpicture}[scale=.70]
        \filldraw[fill=white](0,1)--(1,2)--(3,1)--(2,0)--(0,1);
        \draw (1,1.25) circle (.13);
        \draw (1.5,1) circle (.13);
        \draw (-.5,1.75)--(-.5,-.5);
        \draw (-.5,1) circle (.13);
        
        \node at (1.25,-1.5) {$\wt{\delta}$};
      \end{tikzpicture} \quad \quad\quad
      \begin{tikzpicture}[scale=.70]
        \draw (-.5,1.75)--(-.5,-.5);
        \draw (-1,1.75)--(-1,-.5);
        \draw (-.5,1) circle (.13);
        \draw (-1,.5) circle (.13);
        \draw (-1.5,1) circle (.13); 
        
        \node at (-1,-1.5) {$\wt{\epsilon}$};
      \end{tikzpicture}
      \caption{A picture of a general member of each of the lifts of the MS basis of codimension $4$}
      \label{bigfig3}
    \end{center}
  \end{figure}
\begin{figure}[h]
  \stepcounter{theorem}
  \begin{center}
  \begin{tikzpicture}[scale=.70]
    \filldraw[fill=white](0,1)--(1,2)--(3,1)--(2,0)--(0,1);
    \draw (1,1.25) circle (.13);
    \draw (1,0) circle (.13);
    \draw (-.5,1.75)--(-.5,-.5);
    \draw (-.5,1) circle (.13);
    
    \node at (1,-2) {$\wt{U}$};
    \end{tikzpicture}\quad\quad\quad
    \begin{tikzpicture}[scale=.70]
     \filldraw[fill=white,dashed](0,1)--(1,2)--(3,1)--(2,0)--(0,1);
      \draw[thick] (1,2)--(3,1);
      \draw (1,1.25) circle (.13);
      \draw (1.5,1) circle (.13);
      \draw (-.5,1.75)--(-.5,-.5);

      \draw (-.5,1) circle (.13);

      \node at (1,-2) {$\wt{V}$};
      \end{tikzpicture}\quad\quad\quad
    \begin{tikzpicture}[scale=.70]
      \filldraw[fill=white](0,-1)--(1,0)--(3,-1)--(2,-2)--(0,-1);
      \filldraw[fill=white](0,0)--(1,1)--(3,0)--(2,-1)--(0,0);
      \filldraw[fill=white](0,1)--(1,2)--(3,1)--(2,0)--(0,1);
      \draw (1,1.25) circle (.13);
      \draw (1,0) circle (.13);
      \draw (1,-1) circle (.13);
      \node at (1.5,-2.5) {$\wt{W}$};
    \end{tikzpicture}\[\]
    \begin{tikzpicture}[scale=.70]
      \filldraw[fill=white](0,0)--(1,1)--(3,0)--(2,-1)--(0,0);
      \filldraw[fill=white](0,1)--(1,2)--(3,1)--(2,0)--(0,1);
      \draw (1,1.25) circle (.13);
      \draw (1.5,1) circle (.13);
      \draw (1,0) circle (.13);

      \node at (1.5,-2) {$\wt{X}$};
    \end{tikzpicture}\quad\quad\quad
    \begin{tikzpicture}[scale=.70]
      \filldraw[fill=white](0,1)--(1,2)--(3,1)--(2,0)--(0,1);
      \draw (1,1.25) circle (.13);
      \draw (1.5,1) circle (.13);
      \draw (2,.75) circle (.13);

      \node at (1.5,-2) {$\wt{Y}$};
    \end{tikzpicture}\quad\quad\quad
    \begin{tikzpicture}[scale=.70]
      \filldraw[fill=white](0,1)--(1,2)--(3,1)--(2,0)--(0,1);
      \draw (1,1.25) circle (.13);
      \draw (1.5,1) circle (.13);
      \draw (1,0) circle (.13);
      \draw (.5,.75)--(1.5,1.75);

      \node at (1.5,-2) {$\wt{Z}$};
    \end{tikzpicture}\end{center}
  \caption{The lifts of the MS basis elements of codimension $3$}
  \label{bigfig4}
\end{figure}
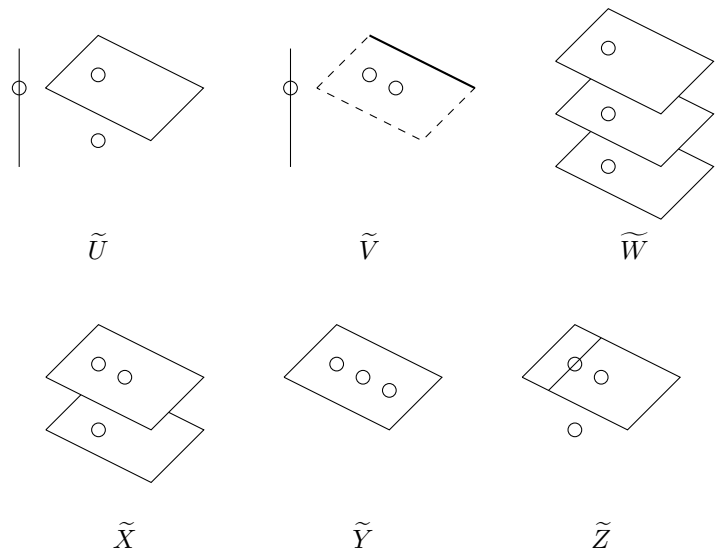
For figuring out the rest of the generators of a potential MS basis in the $\Hilb_3 \PP^3$ case, there is a natural way to extend the basis for cycles of codimensions 2 and 3: we use collinearity conditions. Extending the basis for cycles of dimension 2 and 3 is trickier, so we break from the usual look of the MS basis and add in the duals of the new generators, since we would need to compute them regardless. This does introduce basis elements parametrizing schemes that are generally not reduced, but note that most of these introduced classes are fairly well-behaved when it comes to computations.

We are now able to define the \textit{extended MS basis} (or just ''MS basis'') in (co)dimension 2 in $\Hilb_3 \PP^3$. Consider the basis for $A_7(\Hilb_3 \PP^3)$ given by $\wt{A}, \wt{B}, \wt{C},\wt{D},\wt{E}, M$, where $M$ is the locus of schemes containing a length two subscheme collinear with a fixed point $Q$. Additionally, consider the basis for $A_2(\Hilb_3 \PP^3)$ given by $\alpha,\beta,\gamma,\delta,\epsilon, \mu$, where $\mu$ is the locus parametrizing schemes given by the union of a fixed point $Q$ and a length two scheme supported at a second fixed point $R$. Then the intersection matrix is as follows.
\begin{center}
  \begin{tabular}{ c|cccccc } 
   & $\alpha$ & $\beta$ & $\gamma$ & $\delta$ &$\epsilon$ & $\mu$ \\ 
  \hline
 $\wt{A}$ & 0 & 0 & 1 & 0 & 0 & 0\\
 $\wt{B}$ & 0 & 1 & 2 & 1 & 0 & 0\\
 $\wt{C}$ & 1 & 2 & 2 & 1 & 0 & 0\\
 $\wt{D}$ & 0 & 1 & 1 & 0 & 0& 0\\
 $\wt{E}$ & 0 & 0 & 0 & 0 & 1& 0\\
 $M$ & 0 & 0 & 0 & 0 & 0 & 1
 \end{tabular}
\end{center}
The upper-left $5 \times 5$ comes from the push-pull formula reducing the intersections to a computation in $\Hilb_3 \PP^2$. For the intersections with $M,\mu$: the zeros are clear. To calculate that $M \cdot \mu = 1$, we check transversality in coordinates. Note that this is equivalent to checking the intersection number between $M'$, the locus of 2 points collinear with a fixed point $Q$, and $\mu'$, the locus of nonreduced points supported at a point $Q'$, in $\Hilb_2 \PP^3$.

Pick a copy of $\A^3$ in $\PP^3$ and give it coordinates $x,y,z$. Then we have coordinates $\beta,\gamma,b_1,b_2,c_1,c_2$ on an open subset of $\Hilb_2(\A^3)$ by writing the ideal of a general length 2 subscheme as
\begin{align*}
(x^2+\beta x + \gamma,y-b_1x-b_2,z-c_1x-c_2)
\end{align*}
Set $Q = (0,0,0)$ and $Q' = (1,0,0)$. Then the equations cutting out $M'$ are
\begin{equation*}
b_2 = 0, c_2 = 0,
\end{equation*}
and the equations defining $\mu'$ are
\begin{align*}
b_1 = -b_2, \ c_1 = -c_2, \ \beta = -2, \gamma = 1.
\end{align*}
Since $\dim_{\C} \C[\beta,\gamma,b_1,b_2,c_1,c_2]/(b_2,c_2,b_1+b_2,c_1+c_2, \beta+2,\gamma-1) = 1$, we see that $M \cdot \mu = M' \cdot \mu' = 1$. Note that we are taking the reduced structure here.



\begin{proposition}\label{prop:P3-conversion}
We have the following change of bases between the RL basis and the extended MS basis in $A_2(\Hilb_3 \PP^3)$.
\begin{align*}
P_3 H^2 \ell &= \gamma + \epsilon\\
P_3 H^2 P_2 &= 3\alpha + \beta - \gamma + 2\delta + \epsilon \\
P_3 \ell^2 &= \epsilon\\
P_3 \ell p &= \delta \\
P_3 \ell P_2 &= \alpha \\
P_2 H \ell p &=  2 \beta + 3 \delta + 2 \epsilon + \mu
\end{align*}
Correspondingly, we have the following conversion between the RL and MS bases in $A_7(\Hilb_3 \PP^3)$.
\begin{align*}
P_2 &= \wt{A}\\
P^2 &= 3\wt{A} - \wt{B} + \wt{C} -2 \wt{D} - \wt{E} + M\\
PH &= \wt{B} -\wt{C} + 2\wt{D} + \wt{E} \\
H^2 &= \wt{C} + \wt{E}\\
\ell &= \wt{E}\\
p &= \wt{D}
\end{align*}
\end{proposition}
\begin{proof}
For $A_2(\Hilb_3 \PP^3)$, the first five equations follow from pushing forward the formulas in Proposition \ref{prop:P2-conversion}. For the last equation, note that:
\begin{align*}
    \begin{tabular}{ c|cccccc } 
    & $P_2$ & $P^2$ & $PH$ & $H^2$ & $\ell$ & $p$ \\
    \hline
    $\mu$ & 0 & 1 & 0 & 0 & 0 & 0
   \end{tabular}
\end{align*}
The zeroes in this table can be seen from the cycles being disjoint, and $\mu \cdot P^2 = 1$ follows from $P^2 = P_2 + 
\Al_3\PP^3$. See Equation \eqref{P^2}. Then Table \ref{table:(co)dim2-P3} implies
\begin{align*}
\mu = -2P_3H^2 \ell -2 P_3 H^2 P_2 + 2P_3\ell^2 + P_3\ell p +  6P_3 \ell P_2 + P_2 H \ell p,
\end{align*}
and therefore
\begin{align*}
P_2 H \ell p = 2 \beta + 3 \delta + 2 \epsilon + \mu.
\end{align*}
The $A_7(\Hilb_3 \PP^3)$ conversions follow.
\end{proof}
\subsection{Extending the MS basis in (co)dimension 3.}\label{subsec:extMS-codim3}
We now proceed with giving an extended version of the MS basis for $A_3(\Hilb_3 \PP^3)$ and $A_6(\Hilb_3 \PP^3)$. Note that $\dim A_3(\Hilb_3 \PP^2) = 6$ and $\dim A_3(\Hilb_3 \PP^3) = \dim A_6(\Hilb_3 \PP^3) = 10$, so we need to find four additional basis elements in degree $3$ and $6$.

Let $Q \subseteq L \subseteq \Pi$ be a point, line, and plane, respectively. Let $R$ additionally be a point not on $\Pi$. For the additional $A_6(\Hilb_3 \PP^3)$ basis elements: we define $N_1$ to be the locus of schemes whose generic member is the union of a point on $\Pi$ and a length two scheme collinear with $R$. Let $N_2$ be the locus of schemes containing a length two subscheme contained in $\Pi$ and collinear with $Q$. Let $N_3$ be the locus of schemes containing $Q$. Lastly, let $N_4$ be the locus of schemes coplanar with $L$, and containing a length two subscheme collinear with $Q$.

Then for the $A_3(\Hilb_3 \PP^3)$ basis elements: let $\nu_1$ be the locus of schemes whose generic member is a nonreduced scheme of length $2$ supported at $Q$ and a point on $L$. Let $\nu_2$ be the locus of schemes whose generic member is the union of a nonreduced scheme supported on $L$ with the point $Q$. Let $\nu_3$ be the locus of schemes containing $Q$ and $R$. Lastly, let $\nu_4$ be the locus of schemes that are totally nonreduced, supported at $Q$, and have a length two subscheme contained in $\Pi$. See Figure \ref{fig:co-dim3P3pics} for pictures of a general member of each locus.

\begin{remark}
For $\nu_1,\nu_2,\nu_3,N_1,N_2,N_3$, we take the usual reduced scheme structure. For $\nu_4$, we specifically take the scheme structure gained from writing it as the intersection $\nu_4 = F \cap N_3 \cap \overline{\OO_{2,\nr}}$, where $\overline{\OO_{2,\nr}}$ is the orbit closure parametrizing schemes supported at a single point. The class of $\overline{\OO_{2,\nr}}$ is computed in Section \ref{subsec:O2nr}. This structure on $\nu_4$ is chosen for ease of computation.

For $N_4$, we take the scheme structure gained from writing it as $P \cap M$. It will turn out that this aligns with the reduced structure.
\end{remark}
\begin{figure}[h]
  \stepcounter{theorem}
  \begin{center}
\begin{tikzpicture}[scale=.70]
  \filldraw[fill=white](0,1)--(1,2)--(3,1)--(2,0)--(0,1);
  \draw (1.5,1) circle (.13);

  \node at (0,.5) {$\times$};
  \draw[dashed] (0,.5)--(2,-.5);
  \draw (1,0) circle (.13);
  \draw (1.5,-.25) circle (.13);  
  \node at (1.5,-1) {$N_1$};
\end{tikzpicture}\quad\quad\quad
\begin{tikzpicture}[scale=.70]
  \filldraw[fill=white](0,1)--(1,2)--(3,1)--(2,0)--(0,1);

  \node at (1,1.25) {$\times$};
  \draw[dashed] (1,1.25)--(2.5,.5);
  \draw (1.5,1) circle (.13);
  \draw (2,.75) circle (.13);
  
  \draw (1,0) circle (.13);

  \node at (1.5,-1) {$N_2$};
  \end{tikzpicture}\quad\quad\quad
  \begin{tikzpicture}[scale=.70]
    \draw (-.5,1) circle (.13);
    \draw (-1,.5) circle (.13);
    \filldraw (-1.5,1) circle (.13); 
    \node at (-1,-1) {$N_3$};
  \end{tikzpicture}\quad\quad\quad
  \begin{tikzpicture}[scale=.70]
    \filldraw[fill=white,dashed](0,1)--(1,2)--(3,1)--(2,0)--(0,1);
    \draw[thick] (1,2)--(3,1);
    \node at (1.5,1.75) {$\times$};
    \draw (1.125,1.375) circle (.13);
    \draw (.75,1) circle (.13);
    \draw (2,.75) circle (.13);
    \draw[dashed] (.5,.75)--(1.5,1.75);
    \node at (1.5,-1) {$N_4$};
  \end{tikzpicture}
\vspace{1.5em}

  \begin{tikzpicture}[scale=.70]
    \draw (0,1.75)--(0,-.5);
    \filldraw (0,1) circle (.13);
    \draw (0,0) circle (.13);
    \draw[dashed,-stealth](0,1)--(.5,1.5);

    \node at (0,-1) {$\nu_1$};
  \end{tikzpicture}\quad\quad\quad\quad
  \begin{tikzpicture}[scale=.70]
    \draw (0,1.75)--(0,-.5);
    \filldraw (0,1) circle (.13);
    \draw (0,0) circle (.13);
    \draw[dashed,-stealth](.125,.125)--(.5,.5);
    
    \node at (0,-1) {$\nu_2$};
  \end{tikzpicture}\quad\quad\quad\quad
  \begin{tikzpicture}[scale=.70]
    \draw (-.5,1) circle (.13);
    \filldraw (-1,.5) circle (.13);
    \filldraw (-1.5,1) circle (.13); 
    \node at (-1,-1) {$\nu_3$};
  \end{tikzpicture}\quad\quad\quad\quad
  \begin{tikzpicture}[scale=.70]
    \filldraw[fill=white](0,1)--(1,2)--(3,1)--(2,0)--(0,1);

    \filldraw (1,1.25) circle (.13);

    \draw [domain=0:.75,dashed,-stealth] plot (\x+1,{(\x)^2+1.25});
    \draw[dashed,-stealth](1,1.25)--(1.5,1);

    \node at (1.5,-.5) {$\nu_4$};
    \end{tikzpicture}\end{center}
\caption{Pictures of a general member of $N_1,N_2,N_3,N_4$ and $\nu_1,\nu_2,\nu_3,\nu_4$.}
\label{fig:co-dim3P3pics}
\end{figure}
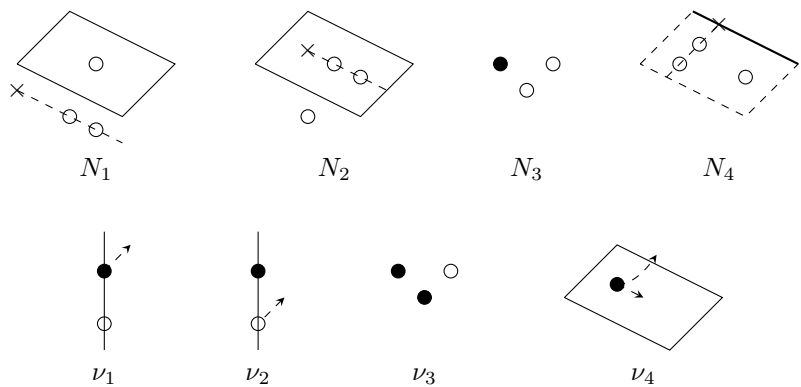
\begin{proposition}\label{prop:(co)dim3-MS-basis-ints} We have that 
  \[\{\wt{U},\wt{V},\wt{W},\wt{X},\wt{Y},\wt{Z}, N_1,N_2,N_3,N_4\}\]
forms a basis of $A_6(\Hilb_3 \PP_3)$ and 
  \[\{U,V,W,X,Y,Z,\nu_1,\nu_2,\nu_3,\nu_4\}\] 
forms a basis of $A_3(\Hilb_3 \PP^3)$. Furthermore, they have the following intersection matrix:
\[
\begin{tabular}{c|cccccc|cccc}
 & $\wt{U}$ & $\wt{V}$ & $\wt{W}$ & $\wt{X}$ & $\wt{Y}$ & $\wt{Z}$ & $N_1$ & $N_2$ & $N_3$ & $N_4$\\
\hline
$U$ & 1 & 1 & 0 & 0 & 0 & 1 & 0 & 0 & 0 & 0\\
$V$ & 1 & 1 & 0 & 0 & 0 & 0 & 0 & 0 & 0 & 0\\
$W$ & 0 & 0 & 6 & 3 & 1 & 0 & 0 & 0 & 0 & 0\\
$X$ & 0 & 0 & 3 & 1 & 0 & 0 & 0 & 0 & 0 & 0\\
$Y$ & 0 & 0 & 1 & 0 & 0 & 0 & 0 & 0 & 0 & 0\\
$Z$ & 1 & 0 & 0 & 0 & 0 & 1 & 0 & 0 & 0 & 0\\
\hline
$\nu_1$ & 0 & 0 & 0 & 0 & 0 & 0 & 1 & 0 & 0 & 0\\
$\nu_2$ & 0 & 0 & 0 & 0 & 0 & 0 & 0 & 1 & 0 & 0\\
$\nu_3$ & 0 & 0 & 0 & 0 & 0 & 0 & 0 & 0 & 1 & 0\\
$\nu_4$ & 0 & 0 & 0 & 0 & 0 & 0 & 0 & 0 & 0 & 9\\
\end{tabular}
\]
Henceforth, we will refer to these bases as the \textit{extended MS bases} or just \textit{MS bases} in dimension 6 and 3, respectively.
\end{proposition}
\begin{proof}
The top-left matrix follows quickly from the definitions of the lifts. The top-right and bottom-left matrices are clearly zero. That leaves the bottom-right matrix.

The off-diagonal entries are clearly zero. For the diagonal entries: the first two diagonal entries are 1 by a similar argument to the computation of $M \cdot \mu$. The third diagonal entry is clearly 1 as the intersection $N_3 \nu_3 = N_3^3$ is transverse and supported at the point corresponding to $\Gamma = Q_1 \cup Q_2 \cup Q_3$ for fixed distinct points $Q_1,Q_2,Q_3$. That leaves the fourth diagonal entry. This can be computed by observing using that
\begin{align*}
N_4 \cdot \nu_4 &= (PM)(F \cdot N_3 \cdot \overline{\OO_{2,\nr}})\\
&=P(P^2 + PH - 3P_2)(F \cdot N_3 \cdot \overline{\OO_{2,\nr}}) \tag{Prop. \ref{prop:P3-conversion}}\\
&= P(P^2 + PH - 3P_2)((P+H)N_3(3(P_2 - PH + \ell + p))) \tag{Lemma \ref{lem:F=P+H}, Subsec. \ref{subsec:O2nr}}\\
&= P(P^2 + PH - 3P_2)((P+H)(3P_3 - P_2H + P \ell)(3(P_2 - PH + \ell + p))) \tag{Lem. \ref{lemma:class-contains-point}}
\end{align*}
Now that we have $N_4 \cdot \nu_4$ in terms of weight 9 monomials in $P,H,P_3,\betarep,P_2,\ell,p$, we can use Tables 1-10 in \cite{RL90} to get that $N_4 \cdot \nu_4 = 9$. Note that none of our cited lemmas in decomposing $N_4 \cdot \nu_4$ rely on knowing the value of $N_4 \cdot \nu_4$.
\end{proof}
\begin{proposition} We have the following conversion between the RL and MS bases in $A_{3}(\Hilb_3 \PP^3)$.
  \begin{align*}
  P_3H \ell &= U \\
  P_3 P_2 H &= -U+V+3Y+Z\\
  P_3^2 &= Y\\
  P_3 \betarep &= X\\
  P_3 H^3 &= 3U + W\\
  P_3 P H^2 &= V-W+4X+2Z \\
  (P_2^2H^2 + P_2 H^2 \ell) &= 2U + 6V + 4X + 9Y + 6Z + 3\nu_1 + \nu_2 + 2\nu_3\\
  P_2 \ell p &= V + Y + Z + \nu_1\\
  P_2 \ell^2 &= 2V + \nu_3\\
  PH\ell p &= 4U + V + 2X + Z + 4\nu_1 + 2\nu_2 + \nu_3 + \tfrac{1}{9}\nu_4
\end{align*}
Correspondingly, we have the following conversion between the RL and MS bases in $A_6(\Hilb_3 \PP^3)$.
\begin{align*}
P_3 &= \wt{Y}\\
P_2P &= \wt{U} - \wt{V} -\wt{Z} + N_4\\
P_2H &= -\wt{U} + \wt{V} + 3\wt{Y} + \wt{Z}\\
PH^2 &= \wt{V} - \wt{W} + 4\wt{X} + 2\wt{Z} + N_3 \\
P^2H &= -3\wt{U} + 2\wt{V} + \wt{W} - 4\wt{X} + 9\wt{Y} + \wt{Z} + N_1 + 2N_2 + N_3 \\
P \ell &= - \wt{U} + \wt{V} + \wt{Z} + N_3\\
P p &= 3 \wt{Y} + N_2\\
H^3 &= 3 \wt{U} + \wt{W} + N_3\\
H \ell &= \wt{U} + N_3\\
\betarep &= \wt{X}
\end{align*}
\end{proposition}
\begin{proof}
The first six equations for $A_3(\Hilb_3 \PP^3)$ come from a computation in $\Hilb_3 \PP^2$. The formulas for $P_3 H \ell, P_3^2, P_3 \betarep, P_3 H^3$ are immediate. Then observing that
\begin{align*}
\begin{tabular}{c|cccccccccc}
 & $\wt{U}$ & $\wt{V}$ & $\wt{W}$ & $\wt{X}$ & $\wt{Y}$ & $\wt{Z}$ & $N_1$ & $N_2$ & $N_3$ & $N_4$\\
\hline
$P_3P_2H$ & 1 & 0 & 3 & 0 & 0 & 0 & 0 & 0 & 0 & 0 \\
$P_3 PH^2$ & 3 & 1 & 6 & 1 & -1 & 2 & 0 & 0 & 0 & 0
\end{tabular}
\end{align*}
yields $P_3 P_2 H = - U + V + Z + 3Y$ and $P_3 P H^2 = V - W + 4X + 2Z$. The above intersection products are a straightforward computation, though note that $P = F-H$ is used in the computation of $P_3PH^2 \cdot \wt{Y}$. 

From this we can solve for $U,V,W,X,Y,Z$ in terms of $P_3H \ell, P_3 P_2 H, P_3^2, P_3 \betarep, P_3 H^3, P_3 P H^2$. We may also use the results of Lemma \ref{lem:nu-int-prods} along with Table \ref{table:(co)dim3-P3} to compute $\nu_1,\nu_2,\nu_3,\nu_4$ in terms of the RL basis.

In the end, we get the following expression for $E_{3,\MS,\RL}$, the matrix whose $i$-th column expresses the $i$-th member of the MS basis for $A_3(\Hilb_3 \PP^3)$ in terms of the RL basis.
\begin{align*}
E_{3,\MS,\RL} = 
\begin{bmatrix}
1 &  5 & -3 & 0 & 0 & -4 & -1 & -1 & -2 &  27\\
0 &  2 &  0 & 0 & 0 & -1 & -1 & -3 &  0 &  81\\
0 & -6 &  0 & 0 & 1 &  3 &  3 &  0 &  0 & -81\\
0 &  4 &  0 & 1 & 0 & -4 & -1 & -1 &  0 &  36\\
0 & -1 &  1 & 0 & 0 &  1 &  0 &  0 &  0 &  0\\
0 & -1 &  0 & 0 & 0 &  1 &  0 &  0 &  0 &  0\\
0 &  0 &  0 & 0 & 0 &  0 &  0 &  1 &  0 & -18\\
0 &  0 &  0 & 0 & 0 &  0 &  1 & -3 &  0 &  18\\
0 &  0 &  0 & 0 & 0 &  0 &  0 & -2 &  1 &  27\\
0 &  0 &  0 & 0 & 0 &  0 &  0 &  0 &  0 &  9
\end{bmatrix}
\end{align*}
Taking the inverse of this yields $E_{3,RL,MS}$, the matrix whose $i$-th column expresses the $i$-th member of the RL basis in terms of the MS basis. 
\begin{align*}  E_{3,\RL,\MS} = 
\begin{bmatrix}
1 & -1 & 0 & 0 & 3 &  0 & 2 & 0 & 2 & 4\\
0 &  1 & 0 & 0 & 0 &  1 & 6 & 1 & 0 & 1\\
0 &  0 & 0 & 0 & 1 & -1 & 0 & 0 & 0 & 0\\
0 &  0 & 0 & 1 & 0 &  4 & 4 & 1 & 0 & 2\\
0 &  3 & 1 & 0 & 0 &  0 & 9 & 0 & 0 & 0\\
0 &  1 & 0 & 0 & 0 &  2 & 6 & 1 & 0 & 1\\
0 &  0 & 0 & 0 & 0 &  0 & 3 & 1 & 0 & 4\\
0 &  0 & 0 & 0 & 0 &  0 & 1 & 0 & 0 & 2\\
0 &  0 & 0 & 0 & 0 &  0 & 2 & 0 & 1 & 1\\
0 &  0 & 0 & 0 & 0 &  0 & 0 & 0 & 0 & \tfrac{1}{9}
\end{bmatrix}
\end{align*}
This yields the first set of equations. For the second set, we can obtain $E_{6,RL,MS}$, the matrix whose $i$-th column expresses the $i$-th member of the RL basis in terms of the MS basis as 
\[((I_{3,6,\MS} E_{3,\MS,\RL})^{\perp})^{-1}I_{3,6,\RL}= E_{6,\RL,\MS}=
\begin{bmatrix}
0 &  1 & -1 &  0 & -3 & -1 & 0 & 3 & 1 & 0\\
0 & -1 &  1 &  1 &  2 &  1 & 0 & 0 & 0 & 0\\
0 &  0 &  0 & -1 &  1 &  0 & 0 & 1 & 0 & 0\\
0 &  0 &  0 &  4 & -4 &  0 & 0 & 0 & 0 & 1\\
1 &  0 &  3 &  0 &  9 &  0 & 3 & 0 & 0 & 0\\
0 & -1 &  1 &  2 &  1 &  1 & 0 & 0 & 0 & 0\\
0 &  0 &  0 &  0 &  1 &  0 & 0 & 0 & 0 & 0\\
0 &  0 &  0 &  0 &  2 &  0 & 1 & 0 & 0 & 0\\
0 &  0 &  0 &  1 &  1 &  1 & 0 & 1 & 1 & 0\\
0 &  1 &  0 &  0 &  0 &  0 & 0 & 0 & 0 & 0
\end{bmatrix}\]
$I_{3,6,MS}$ is the intersection matrix for $A_3(\Hilb_3 \PP^3) \times A_6(\Hilb_3\PP^3)$ under the extended MS basis. It is calculated in Proposition \ref{prop:(co)dim3-MS-basis-ints}. $I_{3,6,RL}$ is the intersection matrix for $A_3(\Hilb_3 \PP^3) \times A_{6}(\Hilb_3 \PP^3)$ using the RL basis, which is given in Table \ref{table:(co)dim3-P3}.
\end{proof}
\section{Analyzing the seven orbits}\label{section:orbits}
There are seven orbits under the $\textup{PGL}_n(\C)$ action on $\PP^3$; We describe them now. The subscripts indicate the codimension of the orbit.
\begin{enumerate}[1)]
\item $\OO_{0}$, the locus consisting of three distinct, non collinear points.
\item $\OO_{1}$, the locus of non-collinear schemes whose support consists of two points.
\item $\OO_{2,\col}$, the locus consisting of three distinct collinear points.
\item $\OO_{2,\nr}$, the locus of curvilinear schemes supported at a single point that are not contained in a line and are not isomorphic to a fat point in a plane.
\item $\OO_{3,\col}$, the locus of collinear schemes supported at two distinct points.
\item $\OO_{4, \col}$, the locus of schemes supported at a single point and contained in a line.
\item $\OO_{4,\textup{max}}$, the locus of schemes isomorphic to a fat point in a plane, i.e. isomorphic to the variety defined by $\mathcal{I}_{p}^2$ where $\mathcal{I}_p$ is the ideal of a point viewed inside of a plane. In particular, their coordinate ring is isomorphic to $\C[x,y]/(x^2,xy,y^2)$.
\end{enumerate}

In order to utilize Lemma \ref{lemma:kleiman-analog}, we need to analyze the classes of the orbit closures in $A(\Hilb_3(\PP^3))$, as well as the Chow groups and positive cones of many of them. Some of the rays of the positive cones are written as intersections with some auxiliary classes-- the computations of such classes is carried out in Appendix \ref{appendix:computations}. For classes parametrizing certain nonreduced schemes we do not always take the reduced structure, and instead give these loci the scheme structure coming from viewing it as an intersection with $\overline{\OO_{1,\nr}}$. This only scales the class by a positive constant, and so does not affect the computation of the nef and effective cones.
\subsection{Planar fat points}\label{subsec:O4max} Note that $\overline{\OO_{4,\max}} =  \OO_{4,\textup{max}}$. To compute its class, note that $\overline{\OO_{2,\nr}} \cdot F^2$ is supported on two components: $\OO_{4,\max}$ and $\overline{\OO_{2,\nr}} \cdot D'$, where $D'$ denotes the locus of schemes $\Gamma \in \Hilb_3 \PP^3$ containing a length two subscheme $\Gamma'$ such that the line containing $\Gamma'$ is incident to two fixed lines $L,L'$. Thus,
\[\ell \OO_4 = \overline{\OO_{2,\nr}} \cdot (F^2 - kD').\]
Before determining $\ell, k$, we first work on rewriting $D'$. Let $\mathbb{G}(1,3)$ denote the Grassmannian of projective lines in projective 3-space. For a subvariety $\Sigma \subseteq \mathbb{G}(1,3)$, we can define subvarieties $Z_{\Sigma}, Z_{\Sigma}'$ of $\Hilb_3 \PP^3$, see Section \ref{appendix:Grassmannian}.

Since $Z'_{\Sigma_1^2} = Z'_{\Sigma_{1,1}} + Z'_{\Sigma_{2,0}}$ (see Construction \ref{construction:Z-sigma'}), we have that
\begin{align*}
D' = Z'_{\Sigma_1^2} = Z'_{\Sigma_{1,1}} + Z'_{\Sigma_{2,0}} = \wt{D} + M.
\end{align*}
Then,
\begin{align*}
  \ell \OO_4 = \overline{\OO_{2,\nr}} \cdot (F^2 - k(\widetilde{D} + M)).
\end{align*}
Now, observe that $\OO_{4,\max} \cdot P_3 \cdot P \cdot H$ is equal to $\OO_{4,\max,\PP^2} \cdot \overline{\OO_{1,\col,\PP^2}} H = 0$. On the other hand,
\begin{align*}
\overline{\OO_{2,\nr}}\cdot F^2 \cdot P_3 P H &= \overline{\OO_{2,\nr,\PP^2}} \cdot \overline{\OO_{1,\col,\PP^2}} \cdot H F^2 = 27\\
\overline{\OO_{2,\nr}}\cdot (\widetilde{D} + M) \cdot P_3PH &= \overline{\OO_{2,\nr,\PP^2}} \cdot \overline{\OO_{1,\col,\PP^2}} \cdot D \cdot H = 9,
\end{align*}
see \cite{RS21} or \cite{EL86}. Here, $\overline{\OO_{2,\nr,\PP^2}}$ denotes the locus of totally nonreduced schemes in $\Hilb_3 \PP^2$, and $\overline{\OO_{1,\col,\PP^2}}$ denotes the locus of collinear schemes in $\Hilb_3 \PP^2$, and the intersection products on the right side are taken in $A_{\bullet}(\Hilb_3\PP^2)$. So, $k = 3$. Furthermore, since
\begin{align*}
\overline{\OO_{4,\max}} \cdot P_3 H^2 &= \overline{\OO_{4,\max,\PP^2}} \cdot H^2 = 9,\\
\overline{\OO_{2,\nr}} \cdot (F^2 - 3(\wt{D} +M))\cdot P_3 H^2 &= \overline{\OO_{2,\nr,\PP^2}} \cdot (F^2 - 3D)H^2 = 27,
\end{align*}
(see Section 5.1 in \cite{RS21} and \cite{EL86}) we have that $\ell = 3$. Later on in Section \ref{subsec:O2nr} we compute that $\overline{\OO_{2,\nr}} = 3(P_2 - PH + \ell + p)$. This yields the following formula.
\begin{align*}
\overline{\OO_{4,\max}} &= (P_2 - PH + \ell + p)((P+H)^2 - 3(\wt{D} + M))\\
&=(P_2 - PH + \ell + p)((P+H)^2 - 3(p + (P^2 + PH - 3P_2))) \tag{Prop \ref{prop:P3-conversion}}\\
&=(P_2 - PH + \ell + p)(-2P^2 - PH + H^2 - 3p + 9P_2).
\end{align*}

To compute the nef and effective cones, we first observe that $\OO_{4,\textup{max}}$ is isomorphic to the flag variety of pairs $(p,\Lambda)$ of a point and a plane in $\PP^3$. Hence we have:
\begin{center}
  \begin{tabular}{ c|l } 
   $k$ & $\dim A_k(\OO_{4,\max})$   \\ 
   \hline 
   5 & 1 \\ 
   4 & 2\\ 
   3 & 3 \\ 
   2 & 3 \\ 
   1 & 2 \\ 
   0 & 1\\
  \end{tabular}
\end{center}
We want to determine generators for the nef and effective cones. Let $Q \subseteq L \subseteq \Pi$ be a fixed point, line, and plane respectively. Let $R$ be an additional point not in $\Pi$. We begin with the curve and fourfold classes. Let $X_{1,1}$ be the locus of planar fat points supported at $Q$ and coplanar with $R$. Let $X_{1,2}$ be the locus of planar fat points supported on $L$ and contained in $\Pi$. Then:
\begin{align*}
X_{1,1} &= \OO_{4,\max} \cdot Y_1 P \tag{see Lemma \ref{lemma:class-contains-point}}\\
X_{1,2} &= \OO_{4,\max} \cdot \widetilde{\alpha} = \OO_{4,\max} \cdot P_3 H.
\end{align*}
Let $X_{4,1}$ be the locus of planar fat points $\Gamma$ coplanar with $Q$. Let $X_{4,2}$ be the locus of planar fat points supported at some point on $\Pi$.
\begin{align*}
X_{4,1} &= \OO_{4,\max} \cdot P\\
X_{4,2} &= \OO_{4,\max} \cdot H
\end{align*}
For surfaces, let $X_{2,1}$ be the locus of planar fat points contained in the plane $\Pi$, and let $X_{2,2}$ be the locus of planar fat points $\Gamma$ supported at a point of $L$ and coplanar with $L$. Lastly, let $X_{2,3}$ be the locus of planar fat points supported at $Q$. See Figure \ref{fig:2folds-O4max} for a picture of a general member of each locus.
\begin{figure}[h]
  \stepcounter{theorem}
  \begin{center}
\begin{tikzpicture}[scale=.70]
  \filldraw[fill=white](0,1)--(1,2)--(3,1)--(2,0)--(0,1);
  \draw (1.5,1) circle (.13);
  \filldraw[opacity=.3] (1.5,1) circle (.19);

  \node at (1.5,-.75) {$X_{2,1}$};
\end{tikzpicture}\quad\quad\quad
\begin{tikzpicture}[scale=.70]
  \filldraw[dashed,fill=white](0,1)--(1,2)--(3,1)--(2,0)--(0,1);
  \draw (1.5,1) circle (.13);
  \filldraw[opacity=.2] (1.5,1) circle (.19);
  \draw (1,.5)--(2,1.5);

  \node at (1.5,-.75) {$X_{2,2}$};
\end{tikzpicture}\quad\quad\quad
\begin{tikzpicture}[scale=.70]
  \filldraw[dashed,fill=white](0,1)--(1,2)--(3,1)--(2,0)--(0,1);
  \filldraw (1.5,1) circle (.13);
  \filldraw[opacity=.2] (1.5,1) circle (.19);

  \node at (1.5,-.75) {$X_{2,3}$};
\end{tikzpicture}\end{center}
\caption{Pictures of a general member of $X_{2,1}, X_{2,2},X_{2,3}$, the generators of $\Nef^3(\OO_{4,\max}) = \Eff_2(\OO_{4,\max})$. The grey "fuzz" on the point is included to highlight that it is a point with multiplicity. The plane in each diagram is the plane in which the ideal of the point is squared.}
\label{fig:2folds-O4max}
\end{figure}
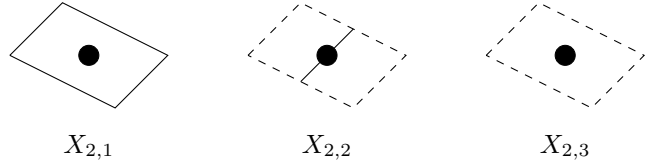
\begin{align*}
X_{2,1} &= \OO_{4,\max} \cdot P_3 = 3(\alpha + \delta + \epsilon) \\
X_{2,2} &= \OO_{4,\max} \cdot Y_3 \\
&= \OO_{4,\max}\cdot (P\ell - Y_1) \tag{Lemma \ref{lemma:class-contains-point}}\\
&= 3(-2 \alpha + \beta + 3 \mu)\\
X_{2,3} &= \OO_{4,\max} \cdot Y_1\\
&= 3(\alpha - \beta + \gamma - \delta - 2\mu) 
\end{align*}
For threefolds, let $X_{3,1}$ be the locus of planar fat points supported at a point of $L$, let $X_{3,2}$ be the locus of planar fat points supported at a point on $\Pi$ and coplanar with $R$, and let $X_{3,3}$ be the locus of planar fat points coplanar with $L$. 
\begin{align*}
X_{3,1} &= \OO_{4,\max} \cdot \ell\\
        &=3(W-3X+3Y-2\nu_1 -2\nu_2 + \tfrac{1}{9}\nu_4)\\
X_{3,2} &= \OO_{4,\max} \cdot H P\\
        &= 3(3U-2V+X-3Y-Z+\nu_1+\nu_2+\tfrac{1}{3}\nu_4)\\
        &=3((3U-2V-W+4X-6Y-Z)+(W-3X+3Y)+\nu_1 + \nu_2 + \tfrac{1}{3}\nu_4)\\
X_{3,3} &= \OO_{4,\max} \cdot P_2\\
&= 4V+X+3Y+2Z+3\nu_1 + 3\nu_3
\end{align*}
\begin{proposition} We have the following expressions for the nef and effective cones of $\OO_{4,\max}$.
\begin{align*}
\Nef^1(\OO_{4,\max}) &= \Eff_4(\OO_{4,\max}) = \langle X_{4,1}, X_{4,2} \rangle\\
\Nef^2(\OO_{4,\max}) &= \Eff_3(\OO_{4,\max}) = \langle X_{3,1}, X_{3,2}, X_{3,3}\rangle\\
\Nef^3(\OO_{4,\max}) &= \Eff_2(\OO_{4,\max}) =  \langle X_{2,1},X_{2,2},X_{2,3}\rangle\\
 \Nef^4(\OO_{4,\max}) &= \Eff_1(\OO_{4,\max}) = \langle X_{1,1}, X_{1,2} \rangle.
\end{align*}
\end{proposition}
\begin{proof}
This follows the emptiness of the apppropriate intersections and from the action of $\textup{PGL}_4(\C)$ being transitive on $\OO_{4,\max}$.
\end{proof} 
\subsection{Collinear schemes supported at a single point} Note that $\overline{\OO_{4,\col}} = \OO_{4,\col}$. It has the class
\begin{align*}
\overline{\OO_{4,\col}} &= \overline{\OO_{2,\col}} \cdot \overline{O_{2,\nr}}\\
&= 3(P_2 - PH + \ell + p)(P^2 - P_2)
\end{align*}
Further, it is isomorphic to the flag variety of pairs $(p,L)$ of a point and a line in $\PP^3$. Hence its Chow groups have the following dimensions.
\begin{center}
  \begin{tabular}{ c|l } 
   $k$ & $\dim A_k(\OO_{4,\col})$   \\ 
   \hline 
   5 & 1 \\ 
   4 & 2\\ 
   3 & 3 \\ 
   2 & 3 \\ 
   1 & 2 \\ 
   0 & 1\\
  \end{tabular}
\end{center}
Let $Q \subseteq L \subseteq \Pi$ be a fixed point, line, and plane respectively. For curve classes, let $X_{1,3},$ be the locus of collinear, totally nonreduced schemes contained in $L$, and let $X_{1,4}$ be the locus of collinear, totally nonreduced schemes $\Gamma$ supported at $Q$ and contained in $\Pi$.
\begin{align*}
X_{1,3} &= \frac{1}{9} \OO_{4,\col} \cdot p^2 = 3(-\phi+2\psi) = 3(-H \ell^2 + 2 H \ell p)\\
X_{1,4} &= \frac{1}{3} \OO_{4,\col} \cdot p\ell = 9(\phi - \psi) = 9(H\ell^2 - H\ell p)
\end{align*}
For fourfolds, let $X_{4,3}$ be the locus of collinear, totally nonreduced schemes supported at a point of $\Pi$, and $X_{4,4}$ be the locus of collinear, totally nonreduced schemes $\Gamma$ that are coplanar with $L$.
\begin{align*}
X_{4,3} &= \OO_{4,\col} \cdot H\\
X_{4,4} &= \frac{1}{3} \OO_{4,\col} \cdot F
\end{align*}
For surfaces, let $X_{2,4}$ be the locus of collinear, totally nonreduced schemes supported at a fixed point $Q$. Let $X_{2,5}$ be the locus of collinear, totally nonreduced scehemes supported at a point of $L$ and contained in $\Pi$. Let $X_{2,6}$ be the locus of collinear, totally nonreduced schemes $\Gamma$ that are collinear with $Q$ and contained in $\Pi$. 
\begin{align*}
X_{2,4} &= \OO_{4,\col} \cdot Y_1 \\
&= 3(2 \alpha - 2 \beta + 2\gamma - 2\delta + 3 \mu)\\
X_{2,5} &= \frac{1}{3} \OO_{4,\col} \cdot H p \\
&= 9(-2\alpha + \beta) \\
X_{2,6} &= \frac{1}{9} \OO_{4,\col} \cdot F p \\
&= 3(\beta - \gamma + 2\delta + \epsilon)
\end{align*}
Lastly, for threefolds, let $X_{3,4}$ be the locus of collinear, totally nonreduced schemes that are collinear with $Q$. Then let $X_{3,5}$ be the locus of collinear, totally nonreduced schemes contained in $\Pi$, and let $X_{3,6}$ be the locus of collinear, totally nonreduced schemes supported at a point of $L$.
\begin{align*}
X_{3,4} &= \frac{1}{3} \OO_{4,\col} \cdot M\\
        &=3(-U+2V+2Y+Z+\nu_1+\nu_3-\tfrac{1}{3}\nu_4)\\
        &=3(U-Z + \nu_1 + \nu_3 - \tfrac{1}{3}) + 6V + 6Y\\
X_{3,5} &= \frac{1}{3}\OO_{4,\col} \cdot p\\
        &= 3(3U-2V-W+4X-6Y-Z)\\
X_{3,6} &= \OO_{4,\col} \cdot \ell\\
        &=3((2W-6X+6Y)+3\nu_1+3\nu_2)
\end{align*}
\begin{proposition} We have the following expressions for the nef and effective cones of $\OO_{4,\col}$.
  \begin{align*}
    \Nef^1(\OO_{4,\col}) &= \Eff_4(\OO_{4,\col}) = \langle X_{4,3}, X_{4,4} \rangle\\
    \Nef^2(\OO_{4,\col}) &= \Eff_3(\OO_{4,\col}) = \langle X_{3,4}, X_{3,5}, X_{3,6}\rangle\\
    \Nef^3(\OO_{4,\col}) &= \Eff_2(\OO_{4,\col}) =  \langle X_{2,4},X_{2,5},X_{2,6}\rangle\\
     \Nef^4(\OO_{4,\col}) &= \Eff_1(\OO_{4,\col}) = \langle X_{1,3}, X_{1,4} \rangle.
    \end{align*}
\end{proposition}
\begin{proof}
  This follows the emptiness of the apppropriate intersections and from the action of $\textup{PGL}_4(\C)$ being transitive on $\OO_{4,\col}$.
\end{proof}
\subsection{Collinear and nonreduced schemes} Note that 
\begin{align*}
\overline{\OO_{3,\col}} = \overline{\OO_{2,\col}} \cdot \overline{\OO_1} &= (P^2-P_2)\cdot 2(H-P)\\
&= 2 (P^2H - P^3 - P_2 H + P_2 P)\\
&= 6P_3 - 6 P_2 P -4 P_2 H + 4 P^2 H
\end{align*}
For the Chow groups, note that $\overline{\OO_{3,\col}}$ is a $\PP^1 \times \PP^1$ bundle over $\mathbb{G}(1,3)$ and as such its Chow groups have the following dimensions:
\begin{center}
  \begin{tabular}{ c|l } 
   $k$ & $\dim A_k(\overline{\OO_{3,\col}})$   \\ 
   \hline
   6 & 1  \\ 
   5 & 3 \\ 
   4 & 5\\ 
   3 & 6 \\ 
   2 & 5 \\ 
   1 & 3 \\ 
   0 & 1\\
  \end{tabular}
\end{center}
Let $Q \subseteq L \subseteq \Pi$ be a point, line, and plane in $\PP^3$ respectively. Let ${\Pi}'$ denote an additional plane in $\PP^3$, and let $L'$ denote an additional line in $\PP^3$.

Let $X_{5,1}$ denote the locus whose generic member is the union of a nonreduced point supported on $\Pi$ and a reduced point, and let $X_{5,2}$ be the locus whose generic member is the union of a reduced point supported on $\Pi$ and a nonreduced point. Furthermore, let $X_{5,3}$ be the locus of collinear, nonreduced schemes that are coplanar with $L$. Also, let $X_{5,4}$ be the locus of totally nonreduced collinear schemes. Then:
\begin{align*}
X_{5,1} &= \overline{\OO_{2,\col}} \cdot Y_2\\
X_{5,2} &= \overline{\OO_{2,\col}} \cdot Y_3 \\
X_{5,3} &= \frac{1}{3} \overline{\OO_{3,\col}} \cdot F\\
X_{5,4} &= \overline{\OO_{2,\nr}} \cdot \overline{\OO_{2,\col}}
\end{align*}
And now let $X_{1,5}$ be the locus of collinear nonreduced schemes contained in $L$ with a length two subscheme supported at $Q$, and let $X_{1,6}$ be the locus whose generic member is the union of $Q$ and a nonreduced scheme of length two contained in $L$. Let $X_{1,7}$ be the locus of collinear, totally nonreduced schemes contained in $\Pi$ and supported at $Q$.

Also, let $X_{1,8}$ be the locus of nonreduced collinear schemes consisting of $Q$ and a nonreduced scheme supported at a point of $L'$, and $X_{1,9}$ be the locus of nonreduced collinear schemes consisting of a nonreduced length two scheme supported at $Q$ and a point on $L'$. Using 5.4 in \cite{RS21}, we see:
\begin{align*}
    X_{1,5} &= 2(-P_3 H \ell^2 + 2P_3 H \ell p)\\
    X_{1,6} &= 4(-P_3 H \ell^2 + 2P_3 H \ell p)\\
    X_{1,7} &= X_{1,4}\\
    X_{1,8} &= 2(P_3 H \ell^2 + P_3 H \ell p)\\
    X_{1,9} &= 2(2 P_3 H \ell^2 - P_3 H \ell p)
\end{align*}
For fourfolds, let $X_{4,5}$ be the locus of collinear nonreduced schemes with a nonreduced length 2 subscheme supported at a point of $L$, and let $X_{4,6}$ be the locus of collinear nonreduced schemes whose generic member is the union of a nonreduced length 2 scheme and a point of $L$. Let $X_{4,7}$ be the locus of schemes whose general member is the union of a reduced point on $\Pi$ and a nonreduced point whose support is in ${\Pi}'$. Let $X_{4,8}$ denote the locus of collinear nonreduced schemes contained in $\Pi$. Lastly, let $X_{4,9}$ denote the locus of collinear nonreduced schemes that are collinear with $Q$.
\begin{align*}
  X_{4,7} &= \overline{\OO_{3,\col}} \cdot \wt{C} = {\OO_{3,\col}} \cdot (H^2 - \ell)\\
  X_{4,8} &= \frac{1}{3} \overline{\OO_{3,\col}} \cdot p\\
  X_{4,9} &= \frac{1}{3} \overline{\OO_{3,\col}} \cdot M
  \end{align*}
Then for surfaces, let $X_{2,7}$ be the locus of collinear nonreduced schemes contained in $\Pi$ and given by the union of $Q$ and a nonreduced length 2 scheme,  and let $X_{2,8}$ be the locus of collinear nonreduced schemes whose generic member is contained in $\Pi$ and consists of a nonreduced scheme supported at $P$ and another point on $\Pi$. Then let $X_{2,9}$ denote the locus of collinear nonreduced schemes contained in $L$. We also recall the definition of $X_{2,5}$, the locus of totally nonreduced schemes contained in $\Pi$ and supported at a point of $L$, and $X_{2,4}$, the locus of totally nonreduced schemes supported at a fixed point $Q$. See Figure \ref{fig:2folds-O3col} for diagrams depicting a general member of each locus.
\begin{figure}[h]
  \stepcounter{theorem}
  \begin{center}
  \begin{tikzpicture}[scale=.70,baseline=6ex]
    \filldraw[fill=white](0,1)--(1,2)--(3,1)--(2,0)--(0,1);
    \draw[dashed](.5,1.5)--(2.5,.5);
      \filldraw (1,1.25) circle (.13);
      \draw[-stealth](1.9,.8)--(1.5,1);
      \draw (2,.75) circle (.13);
    \node at (1.5,-.75) {$X_{2,7}$};
  \end{tikzpicture}\quad\quad\quad
  \begin{tikzpicture}[scale=.70,baseline=6ex]
    \filldraw[fill=white](0,1)--(1,2)--(3,1)--(2,0)--(0,1);
    \draw[dashed](.5,1.5)--(2.5,.5);
      \filldraw (1,1.25) circle (.13);
      \draw[-stealth](1,1.25)--(1.5,1);
      \draw (2,.75) circle (.13);
    \node at (1.5,-.75) {$X_{2,8}$};
    \end{tikzpicture} \quad\quad\quad
    \begin{tikzpicture}[scale=.70,baseline=6ex]
      \draw(.5,1.5)--(2.5,.5);
        \draw (1,1.25) circle (.13);
        \draw (2,.75) circle (.13);
        \draw[thick,-stealth](1.9,.8)--(1.5,1);
      \node at (1.5,-.75) {$X_{2,9}$};
    \end{tikzpicture}\quad\quad\quad
    \begin{tikzpicture}[scale=.70,baseline=6ex]
      \filldraw[fill=white](0,1)--(1,2)--(3,1)--(2,0)--(0,1);
      \draw(.5,1.5)--(2.5,.5);
        \draw (2,.75) circle (.13);
        \draw[ultra thick,dashed,-stealth](2,.75)--(1.75,1.35);
      \node at (1.5,-.75) {$X_{2,5}$};
    \end{tikzpicture}\quad \quad \quad
    \begin{tikzpicture}[scale=.70,baseline=6ex]
        \filldraw (1.5,1) circle (.13);
        \draw[ultra thick,dashed,-stealth](1.5,1)--(.75,1.6);
      \node at (1.5,-.75) {$X_{2,4}$};
    \end{tikzpicture}\end{center}
\caption{Pictures of a general member of $X_{2,7},X_{2,8},X_{2,9},X_{2,5},X_{2,4}$, the generators of $\Eff_2(\overline{\OO_{3,\col}})$. In the pictures for $X_{2,5},X_{2,4},$ a point and a very thick dashed line denotes a \textit{totally nonreduced} collinear scheme, whose direction may vary (possibly within a plane).}
\label{fig:2folds-O3col}
\end{figure}
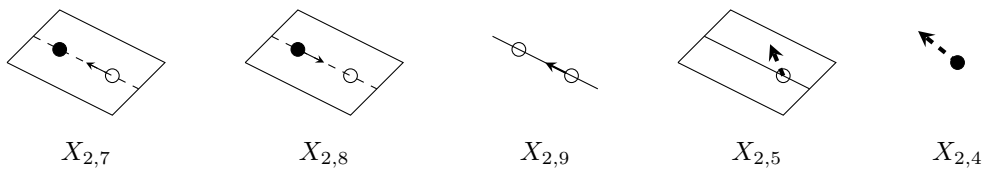
\begin{align*}
X_{2,7} &= 2(-\alpha + \beta - \gamma + 2\delta + \epsilon) \tag{See $S_7$ in \cite{RS21}}\\
X_{2,8} &= 2(-2\alpha + \beta - \gamma + 2\delta + \epsilon)\tag{See $S_9$ in \cite{RS21}}\\
X_{2,9} &= \frac{1}{9}\overline{\OO_{3,\col}} \cdot p^2 = 4 \alpha = 4 P_3 P_2 \ell 
\end{align*}

For threefolds: let $X_{3,7}$ be the locus of collinear nonreduced schemes contained in $\Pi$ whose generic member is a point on $L$ and a nonreduced point in $\Pi$, and let $X_{3,8}$ be the locus of collinear nonreduced schemes contained in $\Pi$ whose generic member is a nonreduced point supported at a point of $L$ and a reduced point. Let $X_{3,9}$ be the locus of collinear nonreduced schemes whose generic member is the union of a reduced point on $L$ and a nonreduced point supported at a point of $\Pi'$, and let $X_{3,10}$ be the locus of collinear nonreduced schemes whose generic member is the union of a nonreduced point supported at a point of $L$ and a reduced point on $\Pi'$. Let $X_{3,11}$ be the locus of collinear nonreduced schemes whose generic member is a nonreducted point supported at $P$ and a freely moving reduced point, and let $X_{3,12}$ be the locus of collinear nonreduced schemes whose generic member consists of $Q$ and a nonreduced point. Finally, let $X_{3,13}$ denote the locus of collinear nonreduced schemes contained in $\Pi$ and collinear with $Q$. Also, recall that $X_{3,5}$ is the locus of totally nonreduced collinear schemes contained in $\Pi$ and $X_{3,6}$ is the locus of totally nonreduced collinear schemes supported on $L$. 
\begin{align*}
X_{3,7} &= 2(2U-V-W+4X-3Y)\tag{see $T_4$ in \cite{RS21}}\\
X_{3,8} &= 2(U-W+4X-6Y+Z) \tag{see $T_5$ in \cite{RS21}}\\
X_{3,11} &= W_1\cdot \overline{\OO_{2,\col}} \tag{see Lemma \ref{lem:W1W2class} for $W_1$}\\
  &= W_1(P^2 - P_2)\\
  &= 2(U-Z + \nu_1 + \nu_3 - \tfrac{1}{3}\nu_4)  \\
X_{3,12} &= W_2 \cdot \overline{\OO_{2,\col}} \tag{see Lemma \ref{lem:W1W2class} for $W_2$}\\
&= W_2(P^2 - P_2)\\
&= 2(U + \nu_1 + \nu_3 - \tfrac{1}{3} \nu_4) \\
X_{3,13} &= \frac{1}{9} \overline{\OO_{3,\col}} \cdot pF\\
&= 2(-2U+2V+3Y+2Z)
\end{align*}
\begin{proposition} We have the following expressions for the nef and effective cones of $\overline{\OO_{3,\col}}$.
\begin{align*}
  \Nef^1(\overline{\OO_{3,\col}}) &= \langle X_{5,1}, X_{5,2}, X_{5,3} \rangle, \quad \Eff_1(\overline{\OO_{3,\col}}) = \langle X_{1,5}, X_{1,6}, X_{1,7}\rangle\\
  \Nef^2(\overline{\OO_{3,\col}}) &= \langle X_{4,5}, X_{4,6}, X_{4,7}, X_{4,8}, X_{4,9}\rangle, \quad \Eff_2(\overline{\OO_{3,\col}}) = \langle X_{2,7}, X_{2,8}, X_{2,9}, X_{2,5}, X_{2,4} \rangle\\
  \Nef^3(\overline{\OO_{3,\col}}) &= \langle X_{3,7},X_{3,8},X_{3,9},X_{3,10},X_{3,11},X_{3,12},X_{3,13}\rangle, \\
   \Eff_3(\overline{\OO_{3,\col}}) &= \langle X_{3,7}, X_{3,8}, X_{3,13}, X_{3,5}, X_{3,11},X_{3,12},X_{3,6}\rangle\\
  \Nef^5(\overline{\OO_{3,\col}}) &= \langle X_{1,5}, X_{1,6}, X_{1,8}, X_{1,9}\rangle,\quad \Eff_5(\overline{\OO_{3,\col}}) = \langle X_{5,1},X_{5,2},X_{5,3},X_{5,4}\rangle
\end{align*}
\end{proposition}
\begin{proof}
For the first two lines, duality follows from the intersection matrix between the generators being diagonal. For $\Nef^5, \Eff_5$, the intersections, taken in $\overline{\OO_{3,\col}}$, between the proposed rays are
\[
\begin{tabular}{c|cccc}
 & $X_{1,5}$ & $X_{1,6}$ & $X_{1,8}$ & $X_{1,9}$\\
 \hline
 $X_{5,1}$ &0& $\ast$& $\ast$ & 0 \\
 $X_{5,2}$ & $\ast$ & 0& 0& $\ast$\\
 $X_{5,3}$ & 0& 0& $\ast$& $\ast$\\
 $X_{5,4}$ & $\ast$ & $\ast$ & 0 & 0\\
\end{tabular}
\]
All the $\ast$ are positive numbers. After scaling the columns and rows in a certain order, we may assume all the nonzero entries are 1. Then relabeling the rows as $e_1,e_2,e_3,e_1+e_2-e_3$, the duality of these cones is equivalent to the duality of $\langle e_1,e_2,e_3,e_1+e_2-e_3\rangle$ and $\langle e_1^*,e_2^*,e_1^* + e_3^*, e_2^*+e_3^* \rangle$, which is clear.


For $\Nef^3,\Eff_3$, we have the following intersection matrix, where we consider the intersections as subschemes in $\overline{\OO_{3,\col}}$.
\[
\begin{tabular}{c|ccccccc}
 & $X_{3,7}$ & $X_{3,8}$ & $X_{3,10}$ & $X_{3,9}$ & $X_{3,11}$ & $X_{3,12}$ & $X_{3,13}$\\
 \hline
 $X_{3,7}$ &0& $\ast$& $\ast$ & 0 & 0 & 0 & 0\\
 $X_{3,8}$ & $\ast$ & 0& 0& $\ast$ & 0 & 0 & 0\\
 $X_{3,13}$ & 0& 0& $\ast$& $\ast$ & 0 & 0 & 0\\
 $X_{3,5}$ & $\ast$ & $\ast$ & 0 & 0 & 0 & 0 & 0\\
 $X_{3,11}$ & 0 & 0 & 0 & 0 & 0 & $\ast$ & 0\\
 $X_{3,12}$ & 0 & 0 & 0 & 0 & $\ast$ & 0 & 0\\
 $X_{3,6}$ & 0 & 0 & 0 & 0 & 0 & 0 & $\ast$\\
\end{tabular}
\]
Once again, after scaling the rows in the columns in the correct order, we may assume all the nonzero intersections are 1. Then a similar argument to the $\Nef^5,\Eff_5$ case yields that the two cones are dual.

Nefness of the proposed nef classes follows from all generators intersecting the suborbits in the correct dimension.
\end{proof}
\subsection{Schemes supported at a single point}\label{subsec:O2nr} We now consider $\overline{\OO_{2,\nr}}$, which parametrizes subschemes supported at a single point (both curvilinear and planar fat points). 

Note that $\overline{\OO_{2,\nr}} \cdot \mu = 0$. Hence $\overline{\OO_{2,\nr}}$ is determined by its restriction to $\Hilb_3 \PP^2$. Then since the locus of totally nonreduced schemes in $\Hilb_3 \PP^2$ has class
\[\overline{\OO_{2,\nr,\PP^2}} = 3(A-B+C-D)\]
by Section 5.3 in \cite{RS21}, we have
\begin{align*}
\overline{\OO_{2,\nr}} &= 3(\widetilde{A}-\widetilde{B}+\widetilde{C}-\widetilde{D})\\
&= 3(P_2 - PH + \ell + p). \tag{Prop \ref{prop:P3-conversion}}
\end{align*}

Since all the classes we consider in this section intersect this orbit in the correct dimension, we do not need to calculate the nef and effective cones of $\overline{\OO_{2,\nr}}$. 

However, for completeness, we give the ranks of the Chow groups. Fix a point $Q$. Consider the following blowup of the punctual Hilbert scheme $\Hilb_3^0 \PP^3$ along the collinear locus.
\[
\begin{tikzcd}
  \wt{\Al^{0}_3} \PP^3 \arrow[hookrightarrow,r,""] \arrow[d,""]& \wt{\Hilb^{0}_3} \PP^3 \arrow[d,""]\\
  \Al^{0}_3 \PP^3 \arrow[hookrightarrow,r,""]& \Hilb^{0}_3 \PP^3
\end{tikzcd}\]
where 
\begin{align*}
\Hilb^{0}_3 \PP^3 &= \{\Gamma \in \Hilb_3 \PP^3 : \Gamma \textup{ is supported at $Q$}\}\\
\wt{\Hilb^0_3}\PP^3 &= \{(\Gamma, \Lambda) \in \Hilb_3^0 \PP^3 \times \PP^2 : \Gamma \subseteq \Lambda\}\\
\Al^{0}_3 \PP^3 &= \{\Gamma \in \Hilb_3 \PP^3: \Gamma \textup{ supported at $Q$ and contained in a line}\}\\
\wt{\Al^{0}_3} \PP^3 &= \{(\Gamma,\Lambda) \in \Hilb^0_3 \PP^3 \times \PP^2 : \Gamma \subseteq \Lambda, \Gamma \in \Al^0_3 \PP^3\}.
\end{align*}
Here, the $\PP^2$ in the definition of $\wt{\Hilb^0_3}\PP^3, \wt{\Al^{0}_3} \PP^3$ parametrizes planes in $\PP^3$ containing $Q$. $\wt{\Hilb^0_3}\PP^3$ is smooth and has a projection to $\PP^2$ where the fibers are copies of the punctual Hilbert scheme $\Hilb^0_3 \PP^2 = \{\Gamma \in \Hilb_3 \PP^2 : \Gamma \textup{ is supported at $Q$}\}$. Observe that, since $\Hilb^0_3 \PP^2$ is isomorphic to a singular quadric surface, all its Chow groups have rank 1. Next we compute that
\[
\begin{tabular}{ c|l } 
$k$ & $\dim A_k(\wt{\Hilb^0_3} \PP^3)$\\
\hline
4& 1\\
3& 2\\
2& 3\\
1& 2\\
0& 1
\end{tabular} \ \ 
\begin{tabular}{ c|l } 
$k$ & $\dim A_k(\wt{\Al^0_3} \PP^3)$\\
\hline
3& 1\\
2& 2\\
1& 2\\
0& 1
\end{tabular}\]
and that we lose one rank from the middle Chow groups of $A_k(\wt{\Hilb^0_3} \PP^2)$ when we blow down. Hence:
\[
\begin{tabular}{ c|l } 
$k$ & $\dim A_k(\Hilb^0_3 \PP^3)$\\
\hline
4& 1\\
3& 1\\
2& 2\\
1& 1\\
0& 1
\end{tabular}
\]
Now, $\overline{\OO_{2,\nr}}$ can be viewed as a $\Hilb^0_3 \PP^3$-bundle over $\PP^3$, yielding the final counts for the ranks of the Chow groups.
\[\begin{tabular}{c|cccccccc}
$k$ & 0 & 1 & 2 & 3 & 4 & 5 & 6 & 7\\
\hline
 $\dim A_k(\overline{\OO_{2,\nr}})$ & 1 & 2 & 4 & 5 & 5 & 4 & 2 & 1
\end{tabular}\]



\subsection{Collinear schemes} The locus of collinear schemes $\overline{\OO_{2,\col}}$ is given equal to $P^2 - P_2$ by Equation \eqref{P^2}. 

As for the nef and effective cones of this orbit closure, note that $\overline{\OO_{2,\col}}$ has the structure of a $\textup{Sym}^3\PP^1 \cong \PP^3$ bundle over $\mathbb{G}(1,3)$. Hence the dimensions of the Chow groups of $\overline{\OO_{2,\col}}$ are
\begin{center}
  \begin{tabular}{ c|l } 
   $k$ & $\dim A_k(\overline{\OO_{2,\col}})$   \\ 
   \hline
   7 & 1 \\ 
   6 & 2  \\ 
   5 & 4 \\ 
   4 & 5\\ 
   3 & 5 \\ 
   2 & 4 \\ 
   1 & 2 \\ 
   0 & 1\\
  \end{tabular}
  \end{center}

Once again let $Q \subseteq L \subseteq \Pi$ denote a point, line,and plane in $\PP^3$ respectively. Let ${\Pi}'$, $\Pi''$ denote additional planes in $\PP^3$. Let $X_{6,1}$ denote the locus of collinear schemes incident to $\Pi$, and let $X_{6,2}$ denote the locus of collnear schemes coplanar with $L$. For new curve classes, let $X_{1,10}$ denote the locus of collinear schemes contained in $L$ and containing two fixed points.
\begin{align*}
X_{6,1} &= \overline{\OO_{2,\col}} \cdot H\\
X_{6,2} &= \frac{1}{3} \overline{\OO_{2,\col}} \cdot F\\
X_{1,10} &= \frac{1}{9} \overline{\OO_{2,\col}} \cdot H^2 p^2 = i_{*}(\alpha)H = H P_3 P_2 \ell
\end{align*}
For fivefolds, let $X_{5,5}$ be the locus of collinear schemes contained in $\Pi$, let $X_{5,6}$ be the locus of collinear schemes whose generic member contains a point in $\Pi \setminus {\Pi}'$ and a point in ${\Pi}' \setminus \Pi$, and let $X_{5,7}$ be the locus of collinear schemes incident to a line. Furthermore, let $X_{5,8}$ denote the locus of schemes collinear with a fixed point $P$. 
\begin{align*}
X_{5,5} &= \frac{1}{3} \overline{\OO_{2,\col}} \cdot p\\
X_{5,6} &= \overline{\OO_{2,\col}} \cdot \wt{C} = \overline{\OO_{2,\col}} (H^2 - \ell)\\
X_{5,7} &= \overline{\OO_{2,\col}} \cdot \ell\\
X_{5,8} &= \frac{1}{3} \overline{\OO_{2,\col}} \cdot M
\end{align*}
For surfaces, there is only one new introduced locus. Let $X_{2,10}$ be the locus of of collinear schemes contained in $L$ and containing $Q$.
\begin{align*}
X_{2,10} &= \frac{1}{9} \overline{\OO_{2,\col}} \cdot Hp^2 = \alpha = P_3 \ell P_2\
\end{align*}
For fourfolds, let $X_{4,10}$ be the locus of collinear schemes contained in $\Pi$ and incident to $L$. Let $X_{4,11}$ be the locus of collinear schemes whose generic member contains a point in $\Pi' \setminus L$ and a point in $L \setminus \Pi'$. Let $X_{4,12}$ be the locus of collinear schemes whose generic member is the union of a point on each of $\Pi, \Pi',$ and $\Pi''$. Let $X_{4,13}$ be the locus of collinear schemes contained in $\Pi$ and collinear with $Q$, and let $X_{4,14}$ be the locus of collinear schemes that contain $Q$.

For threefolds, we have two new classes: let $X_{3,14}$ denote the locus of collinear schemes contained in $\Pi$ and containing $Q$, and let $X_{3,15}$ denote the locus of collinear schemes contained in the fixed line $L$. 
\begin{align*}
  X_{3,14} &= \frac{1}{3}\overline{\OO_{2,\col}}\cdot p \ell \\
  &= -U+V+Z\\
  X_{3,15} &= Y\\
\end{align*}

\begin{proposition}We have the following expressions for the nef and effective cones of $\overline{\OO_{2,\col}}$.
\begin{align*}
\Nef^1 &= \langle X_{6,1}, X_{6,1} \rangle ,\quad \Eff_1 = \langle X_{1,10}, X_{1,4}\rangle \\
\Nef^2 &= \langle X_{5,5}, X_{5,6}, X_{5,7}, X_{5,8}\rangle, \quad \Eff_2 = \langle X_{2,5}, X_{2,10}, X_{2,8}, X_{2,4}\rangle \\
\Nef^3 &= \langle X_{4,10},X_{4,11},X_{4,12},X_{4,13},X_{4,14}\rangle,\quad \Eff_3 = \langle X_{3,5},X_{3,14},X_{3,15},X_{3,6},X_{3,11}\rangle 
\end{align*}
\end{proposition}
\begin{proof}
Duality follows from the generators yielding a diagonal intersection matrix, and the nefness of the proposed generators follows from the fact they intersect all the suborbits in the appropriate dimension.
\end{proof}

\subsection{Non-reduced schemes} The class of the locus of nonreduced schemes $\overline{\OO_{1}}$ can be computed using the push-pull formula to reduce it to a computation in $\Hilb_3\PP^2$. We have that 
\begin{align*}
\overline{\OO_{1}} \cdot P_3 H \ell p &= 2\\ 
\overline{\OO_{1}} \cdot P_3 H \ell^2 &= 0
\end{align*}
hence $\overline{\OO_1} = 2(2H-F) = 2(H-P)$. Since all relevant schemes intersect this orbit in the correct dimension, we do not calculate the effective and nef cones of this orbit closure.
\section{Nef and effective cones in (co)dimension 2,3}\label{section:nef-eff-Hilb3P3}
\begin{theorem}\label{theorem:MAIN(co)dim2-eff-neff} We have the following expression for the nef and effective cones in (co)dimension two for $\Hilb_3 \PP^3$.
\begin{align*}
\Eff_2(\Hilb_3 \PP^3) &= \langle -\alpha + \delta,\ \alpha,\ -\alpha + \beta -\delta,\ \epsilon,\ \mu,\ \alpha - \beta + \gamma -\delta - 2\mu,\ -2\alpha + \beta - \gamma + 2\delta + \epsilon\rangle\\
\Nef^2(\Hilb_3 \PP^3) &= \langle \wt{B},\ \wt{C},\ \wt{D},\ \wt{E},\ \wt{A}+\wt{B},\ \wt{A}+\wt{E},\ 2\wt{A}+2\wt{B}+M,\ 2\wt{A}+2\wt{E}+M\rangle.
\end{align*}
\end{theorem}
\begin{proof}
We first show that these two cones are dual. Set:
\begin{align*}
e_1 &= \alpha - \beta + \gamma - \delta\\
e_2 &= - \alpha + \delta\\
e_3 &= \alpha\\
e_4 &= -\alpha + \beta - \delta\\
e_5 &= \epsilon\\
e_6 &= \mu
\end{align*}
so that $e_1^* = \wt{A}, e_2^* = \wt{B}, e_3^* = \wt{C}, e_4^* = \wt{D}, e_5^* = \wt{E}, e_6^* = M$. Under this basis, our proposed effective note is $C = \langle e_2,e_3,e_4,e_5,e_6,e_1-2e_6,-e_1+e_2+e_5\rangle$. Then 
\begin{align*}
C^{\vee} &= \langle e_2^*, e_3^*, e_4^*, e_5^*, e_6^*, e_1^* + e_2^*, e_1^* + e_5^*, 2e_1^* + 2e_2^* + e_6^*, 2e_1^* + 2e_5^* + e_6^*\rangle \\
&= \langle \wt{B},\ \wt{C},\ \wt{D},\ \wt{E},\ \wt{A}+\wt{B},\ \wt{A}+\wt{E},\ 2\wt{A}+2\wt{B}+M,\ 2\wt{A}+2\wt{E}+M\rangle.
\end{align*}

All the proposed effective classes are indeed effective since we've given effective representations of them in Section \ref{section:orbits} and Lemma \ref{lemma:aux-classes}.

For nefness: note that $\wt{A}$ intersects $\overline{\OO_{2,\col}}, \overline{\OO_{3,\col}}, \overline{\OO_{4,\col}}$ in one higher dimension than expected, and that $\wt{B}, \wt{D},$ and $M$ intersects $\overline{\OO_{4,\max}}$ in one higher dimension than expected. All other intersections between $\wt{A},\wt{B},\wt{C},\wt{D},\wt{E},M$ and the orbit closures are of expected dimension. Then note that  $\wt{A}, \wt{A} + \wt{E}, \wt{A} + \wt{B}, 2\wt{A} + 2\wt{B} + M$, $2\wt{A} + 2\wt{E} + M$,  intersect all the effective surface classes in $\overline{\OO_{2,\col}}, \overline{\OO_{3,\col}}, \overline{\OO_{4,\col}}, \overline{\OO_{4,\max}}$ non-negatively. Furthermore, $\wt{B}, \wt{D}, M$ intersect all the effective surface classes in $\overline{\OO_{4,\max}}$ non-negatively. Therefore, the proposed nef classes are nef by Lemma \ref{lemma:kleiman-analog}.
\end{proof}
\begin{remark} For ease of reading and referencing, we write the rays of each cone in matrix form. The columns of $M_{2,\MS,\Eff}$ give the rays of the effective cone $\Eff_2(\Hilb_3 \PP^3)$ in terms of the extended MS basis, and the columns of $M_{2,\MS,\Nef}$ likewise give the rays of $\Nef^2(\Hilb_3 \PP^3)$ in terms of the extended MS basis.
\begin{align*}
M_{2,\MS,\Eff} &= \begin{bmatrix} 
-1 & 1 & -1 & 0 & 0 &  1 & -2 \\
 0 & 0 &  1 & 0 & 0 & -1 &  1 \\
 0 & 0 &  0 & 0 & 0 &  1 & -1 \\
 1 & 0 & -1 & 0 & 0 & -1 &  2 \\
 0 & 0 &  0 & 1 & 0 &  0 &  1 \\
 0 & 0 &  0 & 0 & 1 & -2 &  0
\end{bmatrix}\\
M_{2,\MS,\Nef} &= \begin{bmatrix} 
0 & 0 & 0 & 0 & 1 & 1 & 2 & 2 \\
1 & 0 & 0 & 0 & 1 & 0 & 2 & 0 \\
0 & 1 & 0 & 0 & 0 & 0 & 0 & 0 \\
0 & 0 & 1 & 0 & 0 & 0 & 0 & 0 \\
0 & 0 & 0 & 1 & 0 & 1 & 0 & 2 \\
0 & 0 & 0 & 0 & 0 & 0 & 1 & 1 
\end{bmatrix}
\end{align*}
\end{remark}
Figure \ref{fig:eff2-figures} provides a picture of a general $\Gamma$ in the effective representatives of the rays of $\Eff_2(\Hilb_3 \PP^3)$.
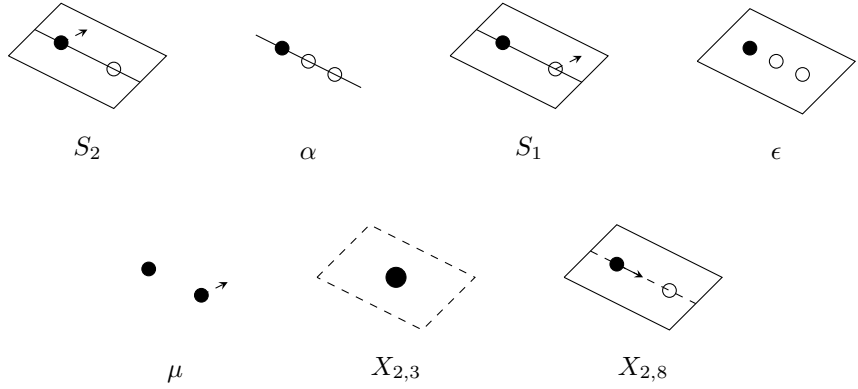
\begin{figure}[h]
  \stepcounter{theorem}
  \begin{center}
    \begin{tikzpicture}[scale=.70]
      \filldraw[fill=white](0,1)--(1,2)--(3,1)--(2,0)--(0,1);
        \draw(.5,1.5)--(2.5,.5);
        \filldraw (1,1.25) circle (.13);
        \draw[dashed,-stealth](1,1.25)--(1.5,1.5);
        \draw (2,.75) circle (.13);
        \node at (1.5,-.75) {$S_2$};
    \end{tikzpicture} \quad \quad \quad
    \begin{tikzpicture}[scale=.70]
      \draw(.5,1.5)--(2.5,.5);
      \filldraw (1,1.25) circle (.13);
      \draw (1.5,1) circle (.13);
      \draw (2,.75) circle (.13);
      \node at (1.5,-.75) {$\alpha$};
    \end{tikzpicture} \quad \quad \quad 
    \begin{tikzpicture}[scale=.70]
      \filldraw[fill=white](0,1)--(1,2)--(3,1)--(2,0)--(0,1);
        \draw(.5,1.5)--(2.5,.5);
        \filldraw (1,1.25) circle (.13);
        \draw (2,.75) circle (.13);
        \draw[dashed,-stealth](2,.75)--(2.5,1);
        \node at (1.5,-.75) {$S_1$};
        \end{tikzpicture} \quad \quad \quad
         \begin{tikzpicture}[scale=.70]
          \filldraw[fill=white](0,1)--(1,2)--(3,1)--(2,0)--(0,1);
          \filldraw (1,1.25) circle (.13);
          \draw (1.5,1) circle (.13);
          \draw (2,.75) circle (.13);
          \node at (1.5,-.75) {$\epsilon$};
        \end{tikzpicture}\[\]
        \begin{tikzpicture}[scale=.70]
          \filldraw (1,1.25) circle (.13);
          \filldraw (2,.75) circle (.13);
          \draw[dashed,-stealth](2,.75)--(2.5,1);
          \node at (1.5,-.75) {$\mu$};
        \end{tikzpicture} \quad \quad \quad
        \begin{tikzpicture}[scale=.70]
          \filldraw[dashed,fill=white](0,1)--(1,2)--(3,1)--(2,0)--(0,1);
          \filldraw (1.5,1) circle (.13);
          \filldraw[opacity=.2] (1.5,1) circle (.19);
        
          \node at (1.5,-.75) {$X_{2,3}$};
        \end{tikzpicture} \quad \quad \quad 
        \begin{tikzpicture}[scale=.70]
          \filldraw[fill=white](0,1)--(1,2)--(3,1)--(2,0)--(0,1);
            \draw[dashed](.5,1.5)--(2.5,.5);
            \filldraw (1,1.25) circle (.13);
            \draw[-stealth](1,1.25)--(1.5,1);
            \draw (2,.75) circle (.13);
            \node at (1.5,-.75) {$X_{2,8}$};
          \end{tikzpicture} 
  \end{center}
\caption{Effective representatives for (integer multiples of) the extremal rays of $\Eff_2(\Hilb_3 \PP^3)$.}
\label{fig:eff2-figures}
\end{figure}

\begin{theorem}\label{theorem:MAIN(co)dim3-eff-neff}
We have the following expression for the nef and effective cones in (co)dimension three for $\Hilb_3 \PP^3$.
\begin{align*}
\Eff_3(\Hilb_3 \PP^3) &= \langle -U+V+Z,\ U-Z,\ Y, \ X-3Y, \ W-3X+3Y, \ U-V, \nu_1,\nu_2,\nu_3,\nu_4,\\
&\quad \quad 3U-2V-W+4X-6Y-Z, \ W - 3X + 3Y -2\nu_1 - 2\nu_2 + \tfrac{1}{9}\nu_4, \\
& \quad \quad U-Z + \nu_1 + \nu_3 - \tfrac{1}{3}\nu_4\rangle \\
\Nef^3(\Hilb_3 \PP^3) &= \langle \wt{U}, \ \wt{V}, \ \wt{W},\ \wt{X},\ \wt{V} + \wt{Y},\ \wt{X} + \wt{Y},\ \wt{Z}, \\
& \quad 2\wt{Y} + \wt{Z},\ 2\wt{V} + 2\wt{Y} + N_1,\ 2\wt{X} + 2\wt{Y} + N_1,\ 2\wt{Y} + \wt{Z}+N_1,\\
& \quad 2\wt{V} + 2\wt{Y} + N_2,\ 2\wt{X} + 2\wt{Y} + N_2,\ 2\wt{Y}+\wt{Z} + N_2, \ N_3,\ \\
& \quad \wt{V} + \tfrac{1}{3}N_4, \ \wt{V} + \wt{Y} + \tfrac{1}{3}N_4,\ N_3 + \tfrac{1}{3}N_4, \ 5\wt{X}+5\wt{Y} + 3N_1 + N_4, 3\wt{V}+3\wt{Y}+2N_2+N_4\\
&\quad 5\wt{V}+N_1+2N_4, \ 10\wt{Y}+5\wt{Z}+6N_1+2N_4, \ 6\wt{V} + N_2 + 2N_4, \ N_1 + 5N_3 + 2N_4 \\
& \quad N_2 + 6N_3 + 2N_4, \ 5\wt{V} + 5\wt{Y} + 4N_1+3N_4\rangle.
\end{align*}
\end{theorem}
\begin{proof} We first show that these two cones are dual. Consider a basis of the ten-dimensional vector space $A_3(\Hilb_3 \PP^3)$ given by
\begin{align*}
  e_1 &= -U + V + Z   & e_6 &= U - V\\
  e_2 &= U - Z        & e_7 &= \nu_1\\
  e_3 &= Y            & e_8 &= \nu_2\\
  e_4 &= X - 3Y       & e_9 &= \nu_3\\
  e_5 &= W - 3X + 3Y  & e_{10} &= \nu_4
\end{align*}
then the duals of these vectors in $A_6(\Hilb_3 \PP^3)$ are
\begin{align*}
  e_1^* &= \wt{U} & e_6^*     &= \wt{Z}\\
  e_2^* &= \wt{V} & e_7^*     &= N_1\\
  e_3^* &= \wt{W} & e_8^*     &= N_2\\
  e_4^* &= \wt{X} & e_9^*     &= N_3\\
  e_5^* &= \wt{Y} & e_{10}^*  &= \tfrac{1}{9}N_4
\end{align*}
Under this basis, our proposed effective cone may be written as 
\begin{align*}
C = \langle e_1,e_2,e_3,e_4,e_5,e_6, e_7, e_8, e_9, e_{10}, e_2+e_4-e_5+2e_6,e_5-2e_7-2e_8+\tfrac{1}{9}e_{10},e_2+e_7+e_9-\tfrac{1}{3}e_{10}\rangle 
\end{align*}
Which, using Macaulay2, we compute to have the following dual cone.
\begin{align*}
C^{\vee} = & \langle e_1^*, \ e_2^*, \ e_3^*,\ e_4^*,\ e_2^* + e_5^*,\ e_4^* + e_5^*,\ e_6^*, \\
& \quad 2e_5^* + e_6^*,\ 2e_2^* + 2e_5^* + e_7^*,\ 2e_4^* + 2e_5^* + e_7^*,\ 2e_5^* + e_6^*+e_7^*,\\
& \quad 2e_2^* + 2e_5^* + e_8^*,\ 2e_4^* + 2e_5^* + e_8^*,\ 2e_5^*+e_6^* + e_8^*, e_9^*,\ \\
& \quad e_2^* + 3e_{10}^*, \ e_2^* + e_5^* + 3e_{10}^*,\ e_9^* + 3e_{10}^*, \ 5e_4^*+5e_5^* + 3e_7^* + 9e_{10}^*, 3e_2^*+3e_5^*+2e_8^*+9e_{10}^*\\
&\quad 5e_2^*+e_7^*+18e_{10}^*, \ 10e_5^*+5e_6^*+6e_7^*+18e_{10}^*, \ 6e_2^* + e_8^* + 18e_{10}^*, \ e_7^* + 5e_9^* + 18e_{10}^* \\
& \quad e_8^* + 6e_9^* + 18e_{10}^*, \ 5e_2^* + 5e_5^* + 4e_7^*+27e_{10}^*\rangle
\end{align*}
Hence the proposed nef and effective cones are dual. All our proposed effective classes are effective, since we've given effective representatives of them in Section \ref{section:orbits} and Lemma \ref{lemma:aux-classes}.

To see that the proposed nef classes are nef, observe that $\wt{U},\wt{W}$, $N_3$ intersect all the orbit closures in the correct dimension. $\wt{Y}$ intersects $\overline{\OO_{2,\col}}, \overline{\OO_{3,\col}}, \overline{\OO_{4,\col}}$ in one dimension higher than expected. $\wt{V}$, $\wt{X}$, $\wt{Z}$, $N_1$, $N_2$. intersect $\overline{\OO_{4,\max}}$ in one dimension higher than expected, and $N_4$ intersects $\overline{\OO_{2,\col}}, \overline{\OO_{3,\col}}, \overline{\OO_{4,\col}}, \overline{\OO_{4,\max}}$ in one dimension higher than expected. All the other intersections between the extended MS basis elements and the orbit closures are of the expected dimension. So, any nonnegative-coefficient sum of $\wt{U},\wt{V},\wt{W},\wt{X},\wt{Y},\wt{Z},$ $N_1,N_2,N_3,N_4$ is a (positive rational multiple of a) class of a subscheme that intersects $\overline{\OO_{2,\col}},\overline{\OO_{3,\col}},$ $\overline{\OO_{4,\col}},$ $\overline{\OO_{4,\max}}$ in at most one dimension greater than expected, and intersects $\overline{\OO_0},\overline{\OO_{1,\nr}},\overline{\OO_{2,\nr}}$ in the expected dimension. By construction, (integer multiples of) the rays generating $\mathcal{C}$ are classes of subschemes that intersect all the threefolds of the $\overline{\OO_{2,\col}},\overline{\OO_{3,\col}}, \overline{\OO_{4,\col}},\overline{\OO_{4,\max}}$ non-negatively and are therefore nef by Lemma \ref{lemma:kleiman-analog}.
\end{proof}
\begin{remark} For ease of reading and referencing, we write the rays of each cone in matrix form. The columns of $M_{3,MS,\Eff}$ give the rays of the effective cone $\Eff_3(\Hilb_3 \PP^3)$ in terms of the extended MS basis. The columns of $M_{3,MS,\Nef}$ give the rays of the nef cone $\Nef^3(\Hilb_3 \PP^3)$ in terms of the extended MS basis.
\begin{align*}
M_{3,\MS,\Eff} &= \begin{bmatrix}
-1 & 1  & 0 & 0  & 0  & 1  & 0 & 0 & 0 & 0 & 3  & 0  & 1\\
1  & 0  & 0 & 0  & 0  & -1 & 0 & 0 & 0 & 0 & -2 & 0  & 0\\
0  & 0  & 0 & 0  & 1  & 0  & 0 & 0 & 0 & 0 & -1 & 1  & 0\\
0  & 0  & 0 & 1  & -3 & 0  & 0 & 0 & 0 & 0 & 4  & -3 & 0\\
0  & 0  & 1 & -3 & 3  & 0  & 0 & 0 & 0 & 0 & -6 & 3  & 0\\
1  & -1 & 0 & 0  & 0  & 0  & 0 & 0 & 0 & 0 & -1 & 0  & -1\\
0  & 0  & 0 & 0  & 0  & 0  & 1 & 0 & 0 & 0 & 0  & -2 & 1\\
0  & 0  & 0 & 0  & 0  & 0  & 0 & 1 & 0 & 0 & 0  & -2  & 0\\
0  & 0  & 0 & 0  & 0  & 0  & 0 & 0 & 1 & 0 & 0  & 0 & 1\\
0  & 0  & 0 & 0  & 0  & 0  & 0 & 0 & 0 & 1 & 0  & \tfrac{1}{9}  & -\tfrac{1}{3}\\
\end{bmatrix}
\end{align*}
\begin{align*}
   M_{3,\MS,\Nef} &=\\
&  \hspace{-1.5em} \begin{bmatrix}
 1 & 0 & 0 & 0 & 0 & 0 & 0 & 0 & 0 & 0 & 0 & 0 & 0 & 0 & 0 & 0 & 0 & 0 & 0 & 0 & 0  & 0  & 0  & 0  & 0  & 0 \\
 0 & 1 & 0 & 0 & 1 & 0 & 0 & 0 & 2 & 0 & 0 & 2 & 0 & 0 & 0 & 1 & 1 & 0 & 0 & 3 & 5  & 0  & 6  & 0  & 0  & 5 \\
 0 & 0 & 1 & 0 & 0 & 0 & 0 & 0 & 0 & 0 & 0 & 0 & 0 & 0 & 0 & 0 & 0 & 0 & 0 & 0 & 0  & 0  & 0  & 0  & 0  & 0 \\
 0 & 0 & 0 & 1 & 0 & 1 & 0 & 0 & 0 & 2 & 0 & 0 & 2 & 0 & 0 & 0 & 0 & 0 & 5 & 0 & 0  & 0  & 0  & 0  & 0  & 0 \\
 0 & 0 & 0 & 0 & 1 & 1 & 0 & 2 & 2 & 2 & 2 & 2 & 2 & 2 & 0 & 0 & 1 & 0 & 5 & 3 & 0  & 10 & 0  & 0  & 0  & 5 \\
 0 & 0 & 0 & 0 & 0 & 0 & 1 & 1 & 0 & 0 & 1 & 0 & 0 & 1 & 0 & 0 & 0 & 0 & 0 & 0 & 0  & 5  & 0  & 0  & 0  & 0 \\
 0 & 0 & 0 & 0 & 0 & 0 & 0 & 0 & 1 & 1 & 1 & 0 & 0 & 0 & 0 & 0 & 0 & 0 & 3 & 0 & 1  & 6  & 0  & 1  & 0  & 4 \\
 0 & 0 & 0 & 0 & 0 & 0 & 0 & 0 & 0 & 0 & 0 & 1 & 1 & 1 & 0 & 0 & 0 & 0 & 0 & 2 & 0  & 0  & 1  & 0  & 1  & 0 \\
 0 & 0 & 0 & 0 & 0 & 0 & 0 & 0 & 0 & 0 & 0 & 0 & 0 & 0 & 1 & 0 & 0 & 1 & 0 & 0 & 0  & 0  & 0  & 5  & 6  & 0 \\
 0 & 0 & 0 & 0 & 0 & 0 & 0 & 0 & 0 & 0 & 0 & 0 & 0 & 0 & 0 & \tfrac{1}{3} & \tfrac{1}{3} & \tfrac{1}{3} & 1 & 1 & 2 & 2 & 2 & 2 & 2 & 3 
\end{bmatrix}
\end{align*}
\end{remark}
\begin{corollary}\label{cor:MAIN(co)dim2-eff-neff}The columns of $M_{2,\RL,\Eff}$ give the rays of $\Eff_2(\Hilb_3 \PP^3)$ in the RL basis and the columns of $M_{2,\RL,\Nef}$ give the rays of $\Nef^2(\Hilb_3 \PP^3)$ in the RL basis, where
\begin{align*}
M_{2,\RL,\Eff} &=\begin{bmatrix}
  0 & 0 &  1 & 0 & -2 &  4 &  0\\
  0 & 0 &  1 & 0 & -2 &  3 &  1\\
  0 & 0 & -2 & 1 &  2 & -3 &  0\\
  1 & 0 & -3 & 0 &  1 & -1 &  0\\
 -1 & 1 & -4 & 0 &  6 & -8 & -5\\
  0 & 0 &  0 & 0 &  1 & -2 &  0
 \end{bmatrix} \end{align*}
 \begin{align*}
M_{2,\RL,\Nef} &= \begin{bmatrix}
  0 &  0 & 0 & 0 &  1 & 1 & -1 & -1\\
  0 &  0 & 0 & 0 &  0 & 0 &  1 &  1\\
  1 &  0 & 0 & 0 &  1 & 0 &  3 &  1\\
  1 &  1 & 0 & 0 &  1 & 0 &  2 &  0\\
 -2 & -1 & 0 & 1 & -2 & 1 & -4 &  2\\
 -2 &  0 & 1 & 0 & -2 & 0 & -4 &  0
\end{bmatrix}
\end{align*}
\end{corollary}
\begin{proof}
This follows from $M_{2,\RL,\Eff} = (E_{2,\RL,\MS})^{-1}M_{2,\MS,\Eff}$ and $M_{2,\RL,\Nef} = (E_{7,\RL,\MS})^{-1}M_{2,\MS,\Nef}$, where $E_{2,\RL,\MS}, E_{7,\RL,\MS}$ are the change of bases given in Proposition \ref{prop:P3-conversion}.
\end{proof}
\begin{corollary}\label{cor:MAIN(co)dim3-eff-neff}The columns of $M_{3,\RL,\Eff}$ give the rays of $\Eff_3(\Hilb_3 \PP^3)$ in the RL basis and the twenty-six columns of $M_{3,\RL,\Nef,1}$ and $M_{3,\RL,\Nef,2}$  give the rays of $\Nef^3(\Hilb_3 \PP^3)$ in the RL basis, where
  \begin{align*}
  M_{3,\RL,\Eff} &= \begin{bmatrix}
   0 &  5 & 0 &  0 & -3 & -4 & -1 & -1 & -2 &  27 &  0 &   4 &  -7 \\
   1 &  1 & 0 &  0 &  0 & -2 & -1 & -3 &  0 &  81 & -3 &  17 & -27 \\
  -3 & -3 & 1 & -3 &  3 &  6 &  3 &  0 &  0 & -81 &  3 & -12 &  27 \\
   0 &  4 & 0 &  1 & -3 & -4 & -1 & -1 &  0 &  36 &  0 &   5 &  -9 \\
   0 & -1 & 0 &  0 &  1 &  1 &  0 &  0 &  0 &   0 &  0 &   1 &  -1 \\
   0 & -1 & 0 &  0 &  0 &  1 &  0 &  0 &  0 &   0 &  1 &   0 &  -1 \\
   0 &  0 & 0 &  0 &  0 &  0 &  0 &  1 &  0 & -18 &  0 &  -4 &   6 \\
   0 &  0 & 0 &  0 &  0 &  0 &  1 & -3 &  0 &  18 &  0 &   6 &  -5 \\
   0 &  0 & 0 &  0 &  0 &  0 &  0 & -2 &  1 &  27 &  0 &   7 &  -8 \\
   0 &  0 & 0 &  0 &  0 &  0 &  0 &  0 &  0 &   9 &  0 &   1 &  -3
  \end{bmatrix}
\end{align*}
\begin{align*}
   M_{3,\RL,\Nef,1} &= \begin{bmatrix}
  -3 & -15 &  6 & 0 & -14 & 1 &  9 & 11 & -28 &  2 & 11 & -31 & -1  \\
   0 &   0 &  0 & 0 &   0 & 0 &  0 &  0 &   0 &  0 &  0 &   0 &  0  \\
   1 &   5 & -2 & 0 &   5 & 0 & -3 & -3 &   9 & -1 & -4 &  10 &  0  \\
   0 &  -1 &  0 & 0 &  -1 & 0 &  1 &  1 &  -1 &  1 &  2 &  -2 &  0  \\
   0 &   0 &  0 & 0 &   0 & 0 &  0 &  0 &   1 &  1 &  1 &   0 &  0  \\
  -1 &  -3 &  2 & 0 &  -3 & 0 &  2 &  2 &  -8 & -2 &  0 &  -6 &  0  \\
   0 &   0 &  0 & 0 &   0 & 0 &  0 &  0 &  -2 & -2 & -2 &   1 &  1  \\
   0 &  -1 &  1 & 0 &  -1 & 0 &  1 &  1 &  -2 &  0 &  1 &  -2 &  0  \\
   1 &   5 & -3 & 0 &   5 & 0 & -4 & -4 &  10 &  0 & -4 &  10 &  0  \\
   0 &   4 &  0 & 1 &   4 & 1 & -4 & -4 &   8 &  2 & -4 &   8 &  2  \\
  \end{bmatrix}\\
  M_{3,\RL,\Nef,2} &=
    \begin{bmatrix}
      8 &  3 & -16    & -15    &  2    &   2 & -51 &  -81 &  49 &  -99 &   9 &   9 &  -79\\
      0 &  0 &   \tfrac{1}{3}  &   \tfrac{1}{3}  &  \tfrac{1}{3}  &   1 &   1 &   2  &   2 &    2 &   2 &   2 &    3 \\
     -3 & -1 &   \tfrac{16}{3} &   \tfrac{16}{3} &  -\tfrac{2}{3} &  -2 &  16 &   26 & -19 &   32 &  -4 &  -4 &   24\\
      1 &  0 &  -1    &  -1    &  0    &   3 &  -3 &   -4 &  11 &   -6 &   1 &   0 &   -1 \\
      0 &  0 &   0    &   0    &  0    &   3 &   0 &    1 &   6 &    0 &   1 &   0 &    4 \\
      2 &  1 &  -3    &  -3    &  1    &  -6 &  -9 &  -17 &  -2 &  -18 &   3 &   6 &  -23 \\
      1 &  0 &   0    &   0    &  0    &  -6 &   2 &   -2 & -12 &    1 &  -2 &   1 &   -8 \\
      1 &  0 &  -1    &  -1    &  0    &   0 &  -3 &   -5 &   5 &   -6 &   0 &   0 &   -5 \\
     -4 &  0 &   5    &   5    &  0    &   0 &  15 &   25 & -20 &   30 &   0 &   0 &   25 \\
     -4 &  0 &   4    &   4    &  0    &   5 &  12 &   20 & -20 &   24 &   0 &   0 &   20 
    \end{bmatrix}
  \end{align*}
\end{corollary}





\appendix
\section{Some computations and auxiliary classes} \label{appendix:computations}
In this section we compute some classes used in the main body of the paper. We begin by computing the MS elements in the RL basis, and then in the second subsection we describe some classes needed in the computations of Section \ref{section:orbits}. 
\subsection{Classes appearing as basis elements}
Recall the definitions of $\nu_1,\nu_2,\nu_3,\nu_4$ from Section \ref{subsec:extMS-codim3}. Figure \ref{fig:co-dim3P3pics} in particular gives pictures of a generic member of each of the loci $\nu_1,\nu_2,\nu_3,\nu_4$.
\begin{lemma}\label{lem:nu-int-prods} We have the following values for the intersections of $\nu_1,\nu_2,\nu_3,\nu_4$ with the RL basis of $A_6(\Hilb_3 \PP^3)$.
  \[
    \begin{tabular}{c|cccccccccc}
      & $P_3$ & $P_2P$ & $P_2 H$ & $PH^2$ & $P^2H$ & $P \ell$ & $P p$ & $H^3$ & $H \ell$ & $\betarep$\\
      \hline 
      $\nu_1$ & 0 & 0 & 0 & 0 & 1 & 0 & 0 & 0 & 0 & 0\\
      $\nu_2$ & 0 & 0 & 0 & 0 & 2 & 0 & 1 & 0 & 0 & 0\\
      $\nu_3$ & 0 & 0 & 0 & 1 & 1 & 1 & 0 & 1 & 1 & 0\\
      $\nu_4$ & 0 & 9 & 0 & 0 & 0 & 0 & 0 & 0 & 0 & 0
     \end{tabular}
     \]
\end{lemma}
\begin{proof}
We begin by computing the intersection products of $\nu_1$ with the RL basis in dimension 6. It is clear that $\nu_2 \cdot \mathcal{B} = 0$ for every basis element that is not $\mathcal{B} = P^2H$, as they have empty intersection.

Write $P^2H = (F^2 - 2FH +H^2)H$, and note that 
\[\nu_1 P^2H = \nu_1(F^2-2FH+H^2)H = \nu_1 F^2 H\]
This reduces to a computation in $\Hilb_2 \PP^3$. Let $F'$ be the locus in $\Hilb_2 \PP^3$ consisting of schemes coplanar with a line $L$, and let $\nu_1'$ be the locus in $\Hilb_2 \PP^3$ consisting of schemes supported at a fixed point $P$. Then:
\[\nu_1 F^2H = (\nu_1')(F')^2.\]

For charts on $\Hilb_2 \PP^3$, we use coordinates $a,b,c,d,e,f$ corresponding to the ideals
\[(x^2+ax+b,y+cx+d,z+ex+f)\]
Then the equations for $\nu_1'$ are $a=b=d=f=0$. If we choose the line in the definition of $F'$ to be the line $x=1,y=z$, then we get the condition $c=e$. For the second copy of $F'$, choose the line to be $x=2, y=-z$ to get the condition $c=-e$. Then:
\begin{align*}
  \nu_1 F^2H = (\nu_1')(F')^2 = \dim_{\C} \frac{\C[a,b,c,d,e,f]}{(a,b,d,f,c+e,c-e)} = 1.
\end{align*}

We now turn our attention to $\nu_2$. Again we compute the intersections with the RL basis in dimension 6. It is clear that $\nu_2 \cdot \mathcal{B} = 0$ for every basis element that is not $\mathcal{B} = P^2H, Pp$. We start by computing 
\begin{align*}
\nu_2 \cdot P^2H  &= \nu_2 \cdot F^2H\\
&= \nu_2 \cdot (Z'_{\Sigma_{1}^2} H ) \tag{see Construction \ref{construction:Z-sigma'}}\\
&= \nu_2 \cdot (Z'_{\Sigma_{1,1}} + Z'_{\Sigma_{2}})H \tag{$\Sigma \mapsto Z'_{\Sigma}$ a group morphism}\\
&=\nu_2(Z'_{\Sigma_2})H.
\end{align*}
Now, $\nu_2(Z'_{\Sigma_2})H$ can be reduced to a computation in $\Hilb_2 \PP^2$. Note that $\nu_2(Z'_{\Sigma_2})$ parametrizes unions of a fixed point on a line $L$ and a nonreduced point supported on that same line that is collinear with a fixed point $R$ not contained in $L$. That is, every scheme parametrized by $\nu_2(Z'_{\Sigma_2})$ is contained in the plane spanned by $L$ and $R$.

Let $L'$ be a fixed line in $\PP^2$. Let $R'$ be a fixed point not on $L'$. Then let $\nu_2'$ denote the locus of nonreduced schemes in $\Hilb_2 \PP^2$ supported at a point of $L'$, and let $H'$ denote the locus of schemes in $\Hilb_2 \PP^2$ incident to a fixed line $L''$. Then,
\[\nu_2 \cdot P^2H = \nu_2\cdot(Z'_{\Sigma_2})H = \nu_2' H'.\]
The class of $\nu_2'$ is given by a pencil of lines through $R'$ intersected with a degenerate conic (specifically, the line $L'$ doubled). Deform the conic to a smooth conic. Then $\nu_2' \cdot H' = 2$, since the conic and the line $L''$ in the definition of $H'$ meet in two points. Thus,
\[\nu_2 \cdot P^2H = \nu_2' H' = 2.\]


For computing $\nu_2 \cdot Pp = \nu_2 \cdot Fp$, this reduces to a computation in $\Hilb_2 \PP^2$. Let $\nu_2''$ denote the locus in $\Hilb_2 \PP^2$ of schemes supported at a fixed point $Q$. Let $F''$ denote the locus of schemes in $\Hilb_2 \PP^2$ collinear with a fixed point $R$. Then:
\[\nu_2 \cdot Pp = \nu_2 \cdot Fp = \nu_2'\cdot F''\]
For a chart on $\Hilb_2 \PP^2$, we use coordinates $a,b,c,d$ corresponding to 
\[(x^2+ax+b,y+cx+d).\]
Let $Q = (0,0)$ and $R = (1,1)$. Similar to the computation of $M \cdot \mu$, we have that $\nu'$ is cut out by $a=b=d=0$, and $F''$ is cut out by $1+c+d=0$. Then 
\[\nu_2 \cdot Pp = \nu_2'\cdot F'' = \dim_{\C} \frac{\C[a,b,c,d]}{(a,b,d,1+c+d)}=1.\]
In summary, we have shown:
\[
\begin{tabular}{c|cccccccccc}
  & $P_3$ & $P_2P$ & $P_2 H$ & $PH^2$ & $P^2H$ & $P \ell$ & $P p$ & $H^3$ & $H \ell$ & $\betarep$\\
  \hline 
  $\nu_1$ & 0 & 0 & 0 & 0 & 1 & 0 & 0 & 0 & 0 & 0\\
  $\nu_2$ & 0 & 0 & 0 & 0 & 2 & 0 & 1 & 0 & 0 & 0
 \end{tabular}\]

For $\nu_3, \nu_4$: The third row is a straightforward computation, and can also be computed by realizing that $\nu_3 = N_3^2 = (3P_3 - P_2H + P \ell)^2$ (Lemma \ref{lemma:class-contains-point}). The fourth row follows from $\nu_4 = F\cdot N_3 \cdot \overline{\OO_{2,\nr}}$.
\end{proof}
\subsection{Classes needed to compute positive cones}
\begin{lemma}\label{lemma:class-contains-point}Let $Y_1$ be the locus of $\Gamma \in \Hilb_3 \PP^3$ containing a fixed point $p$. Let $Y_2$ denote the locus of $\Gamma \in \Hilb_3$ with a length two subscheme contained in a fixed line $L$. Furthermore, let $Y_3$ be the locus of schemes that are incident to a line $L$ and coplanar with $L$. Then:
\begin{align*}
  Y_1 &= 3P_3 - P_2H + P\ell\\
  Y_2 &= -P_3P - P_3H + P_2p\\
  Y_3 &= P \ell - Y_1 = -3P_3 + P_2H
\end{align*}
\end{lemma}
\begin{proof}
 $Y_1$ can be computed by computing the following intersection numbers.
\begin{align*}
  \begin{tabular}{c|cccccccccc}
   & $P_3 H \ell$ & $P_3 P_2 H $ & $P_3^2$ & $P_3 \betarep$ & $P_3 H^3$ & $P_3 P H^2$ & $(P_2^2 H^2 + P_2 H^2 \ell)$ & $P_2 \ell p$ & $P_2 \ell^2$ & $P H \ell p$ \\
   \hline 
   $Y_1$ & 0 & 0 & 0 & 0 & 0 & 0 & 2 & 0 & 1 & 1
  \end{tabular}
\end{align*}
We now compute the products of $Y_2$ with the classes of the RL basis of complimentary codimension.
\begin{align*}\hspace{-.5cm}
  \setlength{\tabcolsep}{5pt}
  \begin{tabular}{c|ccccccccccccc}
    & $P_3 H P$ & $P_3 H^2$ & $P_3 \ell $ & $P_3 P_2$ & $P_3 p$ & $P_2 H \ell$ & $P_2^2 H$ & $P_2 \betarep$ & $P_2 H^3$ & $P_2 P H^2$ & $P H^2 \ell$ & $P \ell^2$ & $H \ell p$ \\
    \hline 
    $Y_2$ & 0 & 0 & 0 & 0 & 0 & 0 & 0 & 0 & 0 & 0 & 1 & 0 & 1
   \end{tabular}
\end{align*}
Then using Table \ref{table:(co)dim4-P3}, this yields $Y_2 = -P_3P - P_3 H + P_2 p$. For $Y_3$, consider $P \cdot \ell$, the locus of schemes incident to a line and coplanar with a fixed point. By specializing the fixed point onto the line, we see that
\[P\ell = c_1 Y_3 + c_2 Y_1\]
Intersecting both sides of the equation with $P_3 \ell H$ we see that $c_1 = 1$, and intersecting with $Y_1^2$, we see that $c_2 = 1$. Hence we are done.
\end{proof}
\begin{lemma}\label{lemma:aux-classes}Let $Q \subseteq L \subseteq \Pi$ be a point, line, and plane respectively. Let $T_1$ denote the locus of schemes whose general member is the union of $Q$ and a nonreduced scheme contained in $\Pi$. Let $T_2$ denote the locus of schemes whose general member is contained in $\Pi$ and is the union of a nonreduced length two scheme supported at $Q$ and a reduced point on $\Pi$. Let $T_3$ denote the locus of schemes contained in $\Pi$ that are nonreduced and whose support is contained in $L$. Note that these are threefolds. Then they have the following classes
\begin{align*}
T_1 &= 2(U-V)\\
T_2 &= 2(U-Z)\\
T_3 &= k(X-3Y),
\end{align*}
where $k$ is some positive integer. Let $S_1$ be the locus of schemes contained in $\Pi$ whose general member is a union of $Q$ and a nonreduced point supported at a point of $L$. Let $S_2$ be the locus of schemes contained in $\Pi$ whose general member is a union of a nonreduced point supported at $Q$ and a point of $L$. Then,
\begin{align*}
S_1 &= 2(-\alpha + \beta - \delta) \\
S_2 &= 2(-\alpha + \delta)
\end{align*}
\end{lemma}
\begin{proof}
These results can be found in \cite{RS21}. For the first two equations, see the computation of $T_6,T_7$ in Section 5.4. For the third equation, see the proof of Theorem 6. For the two surface classes, see $S_5,S_{10}$ in Section 5.4.
\end{proof}
\begin{lemma}\label{lem:W1W2class}Fix a point $Q$. Let $W_1$ be the locus in $\Hilb_3 \PP^3$ whose general member is the union of a nonreduced length two scheme supported at $Q$ and a reduced point. Let $W_2$ be the locus in $\Hilb_3 \PP^3$ whose general member is the union of $Q$ and a nonreduced length 2 scheme. 
\begin{align*}
  W_1 &= 2(-5P_3 P + 5P_3 H - P_2 H^2 + P_2HP - 2P_2 \ell + 2P_2 p - (P^2H^2 + PH^3) + 5PH\ell + 4P\betarep)\\
  W_2 &=2(5P_3P -P_3H -P_2HP + P_2\ell - 2P_2 p + (P^2H^2 + PH^3) - 4PH\ell - 4P\betarep)
\end{align*}
Here the scheme structure is determined by $W_1 + W_2 = \overline{\OO_{1,\nr}} \cdot N_3$, meaning we are not necessarily taking the reduced structure.
\end{lemma}
\begin{proof} We first focus on $W_1$. We compute its intersections with RL basis members in $A_4(\Hilb_3 \PP^3)$. For any of the basis elements involving a $P_3$ factor, clearly their intersection with $W_1$ is zero. Furthermore,
\[W_1 \cdot P_2 H \ell = W_1 \cdot P_2 \betarep = W_1 \cdot P_2 H^3 = W_1\cdot PH^2\ell = W_1 \cdot P\ell^2 = W_1 \cdot H \ell p = 0\]
since none of $P_2 H \ell, P_2 \betarep, P_2 H^3, PH^2\ell, P\ell^2, H \ell p$ contain schemes with length two supported at $Q$. That leaves computing $W_1 \cdot P_2^2 H, W_1 \cdot P_2 P H^2$. 

Let $\Pi$ be a fixed plane containined $Q$. Now, note that $W_1 \cdot P_2$ is some multiple of the locus whose general member is contained in $\Pi$ and consists of a length two scheme supported at $Q$ union with a nonreduced point. That is,
\begin{align*}
W_1 \cdot P_2 = k_1\left(\begin{tikzpicture}[scale=.70,baseline=4ex]
  \filldraw[fill=white](0,1)--(1,2)--(3,1)--(2,0)--(0,1);
    \filldraw (1,1.25) circle (.13);
    \draw[dashed,-stealth](1,1.25)--(1.5,1.5);
    \draw (2,.75) circle (.13);
    \end{tikzpicture}\right) = k_1(2(U-Z))
\end{align*}
for some positive constant $k_1$. Then,
\begin{align*}
W_1 \cdot P_2^2H = 2k_1(U-Z)P_2H = 2k_1(U-Z)(-\wt{U}+\wt{V}+3\wt{Y}+\wt{Z}) = 2k_1\\
W_1 \cdot P_2PH^2 = 2k_1(U-Z)PH^2 = 2k_1(U-Z)(\wt{V} -\wt{W} + 4\wt{X} + 2\wt{Z} + N_3) = 2k_1
\end{align*}
We now turn our attention to $W_2$. Similar to $W_1$, we have that the product of $W_2$ and any RL basis element of $A_4(\Hilb_3 \PP^3)$ involving $P_3$ is zero. Likewise,
\[W_2 \cdot P_2 H \ell = W_2 \cdot P_2 \betarep = W_2 \cdot P_2 H^3 = W_2\cdot PH^2\ell = W_2 \cdot P\ell^2 = W_2 \cdot H \ell p = 0\]
as none of $P_2 H \ell, P_2 \betarep, P_2 H^3, PH^2\ell, P\ell^2, H \ell p$ have enough flexibility of the support to both containg $Q$ and a nonreduced length two scheme supported not at $Q$. Now note that
\[W_2 \cdot P_2 = k_2T_6 = k_2(2(U-V)).\]
for some positive constant $k_2$. Then,
\begin{align*}
W_2 \cdot P_2^2H = 2k_2(U-V)P_2H = 2k_2(U-V)(-\wt{U}+\wt{V}+3\wt{Y}+\wt{Z}) = 2k_2\\
W_2 \cdot P_2PH^2 = 2k_2(U-V)PH^2 = 2k_2(U-V)(\wt{V} -\wt{W} + 4\wt{X} + 2\wt{Z} + N_3) = 4k_2
\end{align*}.
So, we have determined $W_1,W_2$ up to scalar. Then using that $N_3 = 3P_3 - P_2H + P \ell$ (see Lemma \ref{lemma:class-contains-point}) and $\overline{\OO_{1,\nr}} = 2(H-P)$, we can compute the following.

\[\resizebox{6in}{!}{$\renewcommand{\arraystretch}{.7}
  \hspace{-1cm}\begin{tabular}[h]{c|cccccccccccccc}
  & $P_3 HP$ & $P_3 H^2$ & $P_3 \ell$ & $P_3 P_2$ & $P_3 p$ & $P_2 H \ell$ & $P_2^2 H$ & $P_2 \betarep$ & $P_2 H^3$ & $P_2 P H^2$ & $PH^2 \ell$ & $P \ell^2$ & $H \ell p$\\
  \hline
  $W_1$ & 0 & 0 & 0 & 0 & 0 & 0 & $2k_1$ & 0 & 0 & $2k_1$ & 0 & 0 & 0 \\
  $W_2$ & 0 & 0 & 0 & 0 & 0 & 0 & $2k_2$ & 0 & 0 & $4k_2$ & 0 & 0 & 0 \\
  $\overline{\OO_{1,\nr}} \cdot Y_1$ & 0 & 0 & 0 & 0 & 0 & 0 & 4 & 0 & 0 & 6 & 0 & 0 & 0 
\end{tabular}$}\]
Hence $k_1 = k_2 = 1$.
\end{proof}
\section{Classes coming from the Grassmannian, and some intersection products}\label{appendix:Grassmannian}
We introduce some classes defined in terms of classes in $\mathbb{G}(1,3)$. We use these classes to compute a few intersection products to correct some errors in \cite{RL90}. We also use them in Lemma \ref{lem:nu-int-prods} and in computing $\overline{\OO_{4,\max}}$ in Section \ref{subsec:O4max}.

Let $Q \subseteq L \subseteq \Pi \subseteq \PP^3$ be a complete flag in $\PP^3$. We fix the following notation for the Schubert cycles on $\mathbb{G}(1,3)$, the Grassmannian of lines in $\PP^3$. See Section 3.3 in \cite{EH16} for a more complete discussion of the Chow ring of $\mathbb{G}(1,3)$.
\begin{align*}
\Sigma_0 &= \mathbb{G}(1,3)\\
\Sigma_{1,0} &= \Sigma_{1} = \{\Lambda \in \mathbb{G}(1,3) : \Lambda \cap L \ne \varnothing \}\\
\Sigma_{2,0} &= \Sigma_{2} = \{\Lambda : Q \in \Lambda\}\\
\Sigma_{1,1} &= \{\Lambda : \Lambda \subseteq \Pi\}\\
\Sigma_{2,1} &= \{\Lambda : Q \subseteq \Lambda \subseteq \Pi\}\\
\Sigma_{2,2} &= \{\Lambda: \Lambda = L\}
\end{align*}
\begin{construction}\label{construction:Z-sigma}
We can define a map on Chow groups:
\begin{align*}
A^{\bullet}(\mathbb{G}(1,3)) &\to A^{\bullet + 2}(\Hilb_3 \PP^3)\\
\Sigma &\mapsto Z_{\Sigma}
\end{align*}
where, for a subvariety $\Sigma \subseteq \mathbb{G}(1,3)$, $Z_{\Sigma}$ is the locus of schemes $\Gamma \in \Hilb_3 \PP^3$ that are contained in some $\Lambda \in \Sigma$. The map is then extended linearly. Note that $Z_{\mathbb{G}(1,3)} = \Al_3 \PP^3,$ the collinear locus of $\Hilb_3 \PP^3$.
\end{construction}

\begin{construction}\label{construction:Z-sigma'}We have a map on Chow groups:
\begin{align*}
  A^{\bullet}(\mathbb{G}(1,3)) &\to A^{\bullet}(\Hilb_3 \PP^3)\\
  \Sigma &\mapsto Z'_{\Sigma},
\end{align*}
  where, for a subvariety $\Sigma \subseteq \mathbb{G}(1,3)$, $Z'_{\Sigma}$ is defined as follows.
  \begin{align*}
  Z'_{\Sigma} = \{\Gamma \in \Hilb_3 \PP^3 : \exists \Gamma' \subseteq \Gamma \textup{ of length two such that } \Gamma' \subseteq \Lambda \textup{ for some line }\Lambda \in \Sigma\}.
\end{align*}
We get the full map by extending linearly.
\end{construction}

We now compute a few intersection products as corrections to some values in Tables 1-9 in \cite{RL90}. We will need a few relations in $A(\Hilb_3 \PP^3)$ involving the classes introduced above for our computations.

\begin{lemma}The following identities hold in $A(\Hilb_3 \PP^3)$.
\begin{align}
 P^2 &= P_2 + Z_{\mathbb{G}(1,3)} \label{P^2}\\
 P_2P &= P_3 + Z_{\Sigma_1} \label{P_2P}\\
 P_2^2 &= Z_{\Sigma_1^2} = Z_{\Sigma_{1,1}} + Z_{\Sigma_{2}} \label{P_2^2}\\
 P^3 &= 4P_2P + P_2H - P^2H - 3P_3 \label{P^3}\\
 P_2^3 &= 5P_3^2 - 2P_3P_2H \label{P_2^3}
\end{align}
\end{lemma}
\begin{proof}
 See Lemma 4.1 (for the first three equations) and Lemma 4.6 (for the latter two) in \cite{RL90}.  
\end{proof}
These equations are either proven using set-theoretic decompositions and then computing the multiplicity, or by substitutions coming from the aforementioned method. As such, we do not need to worry about the next computations relying on intersection products with potential errors of their own.

\begin{lemma}\label{lem:corrections} The following expressions hold in $A_{0}(\Hilb_3 \PP^3)$.
\begin{align*}
P P_2^4 = -2\\
P P_2^{3} \ell = 2\\
PH^2 P_2^2p = -3\\
P^3 p^3 = 1.
\end{align*}
\end{lemma}
\begin{proof}
We prove the equations in order.
\begin{itemize}
\item For the first equation:
\begin{align*}
P P_2^4 &= P P_2 (5P_3^2 - 2P_3 P_2 H) \tag{\eqref{P_2^3}}\\
&= -2 P P_2 \cdot P_3 P_2 H \tag{$P_3^2 P_2 = 0$}\\
&= -2(P_3 P_2 H) \cdot (P_2 P) = -2(1) = -2\tag{Table \ref{table:(co)dim3-P3}}
\end{align*}
Alternatively, one can compute $-2(P_3 P_2 H) \cdot (P_2 P)$ by recognizing it as the following computation in $\Hilb_3 \PP^2$.
\begin{align*}
-2(Hg)(g(F-H)) &= -2(Z-U+V+3Y)(3Y - (Z-U+V+3Y))\\
&= -2(Z-U+V+3Y)(-Z+U-V)\\
&= -2(1) = -2
\end{align*}
\item For the second equation:
\begin{align*}
P P_2^3 \ell &= P \ell (5P_3^2 - 2P_3 P_2 H) \tag{\eqref{P_2^3}}\\
&= 5(F-H)\ell P_3^2 - 2 (F-H)\ell P_3 P_2 H \\
&= -2(F-H)\ell P_3 P_2 H \tag{$P_3^2 \ell  = 0$}\\
&= 2H \ell P_3 P_2 H \tag{$(P_3 P_2 F) \ell = 3(P_3^2) \ell = 0$}\\
&= 2 H^2 \cdot \alpha = 2
\end{align*}
\item For the third equation, note that $P_2^2$ is the locus of collinear schemes whose spanned line is incident to two fixed lines $L,L'$. Then $P_2^2 p = 3 P_3^2$, and 
\[PH^2 P_2^2 p = (F-H)H^2 (P_2^2 p) = 3(F-H)H^2P_3^2 = -3H^3P_3^2 = -3.\]
\item For the fourth equation, we start by observing that
\begin{align*}
P^3p^3 = (4P_2 P + P_2 H - P^2 H - 3P_3)p^3 \tag{\eqref{P^3}}
\end{align*}
Note that $P_2 P$ is the sum of $P_3$ and $Z_{\Sigma_{1}}$, the locus of collinear schemes whose spanned line is incident to a fixed line $L$. But the product $Z_{\mathbb{G}(1,3)}p^3 =0$, so
\begin{align*}
P^3 p^3 &= 4P_2 Pp^3 + (P_2 H - P^2 H - 3P_3)p^3 \\
&= 4(P_3 + Z_{\Sigma_{1}})p^3 + (P_2H - P^2H - 3P_3)p^3 \tag{\eqref{P_2P}}\\
&= 4P_3p^3 + (P_2H - P^2H - 3P_3)p^3\\
&= 4P_3p^3 + (P_2H - (P_2 + Z_{\mathbb{G}(1,3)})H - 3P_3)p^3 \tag{\eqref{P^2}} \\
&= P_3 p^3 - Z_{\mathbb{G}(1,3)}Hp^3\\
&= P_3p^3 = 1
\end{align*}
\end{itemize}
\end{proof}
\begin{remark}It is worth noting that $PH^2P_2^2p$ is correctly listed as $-3$ in Table IV of \cite{RL90} but then incorrectly listed as $3$ later on in Table 5.
\end{remark}

\begin{remark} It is worth noting that one could compute the conversion between the MS and RL bases by instead computing the intersection of the MS basis in dimension 6 with the RL basis in dimension 3 to get $E_{6,\RL,\MS}$. These intersection numbers are computable, though the intersection table has far fewer zeroes and the computations are more involved, often turning into large case work and utilizing the $Z_{\Sigma}, Z'_{\Sigma}$ constructions. 
\end{remark}

\bibliographystyle{alpha}
\bibliography{my}

\begin{thebibliography}{ABCH13}

\bibitem[ABCH13]{ACBH13}
Daniele Arcara, Aaron Bertram, Izzet Coskun, and Jack Huizenga.
\newblock The minimal model program for the {H}ilbert scheme of points on
  {$\Bbb{P}^2$} and {B}ridgeland stability.
\newblock {\em Adv. Math.}, 235:580--626, 2013.

\bibitem[Arb21]{A21}
Noah Arbesfeld.
\newblock K-theoretic {D}onaldson-{T}homas theory and the {H}ilbert scheme of
  points on a surface.
\newblock {\em Algebr. Geom.}, 8(5):587--625, 2021.

\bibitem[BB73]{BB73}
A.~Bia\text{\l}ynicki-Birula.
\newblock Some theorems on actions of algebraic groups.
\newblock {\em Ann. of Math. (2)}, 98:480--497, 1973.

\bibitem[BS17]{BS17}
Dori Bejleri and David Stapleton.
\newblock The tangent space of the punctual {H}ilbert scheme.
\newblock {\em Michigan Math. J.}, 66(3):595--610, 2017.

\bibitem[BZ23]{BZ23}
Dori Bejleri and Gjergji Zaimi.
\newblock The topology of equivariant {H}ilbert schemes.
\newblock {\em Res. Math. Sci.}, 10(3):Paper No. 28, 23, 2023.

\bibitem[CC15]{CC15}
Dawei Chen and Izzet Coskun.
\newblock Extremal higher codimension cycles on moduli spaces of curves.
\newblock {\em Proc. Lond. Math. Soc. (3)}, 111(1):181--204, 2015.

\bibitem[CLO16]{CLO16}
Izzet Coskun, John Lesieutre, and John~Christian Ottem.
\newblock Effective cones of cycles on blowups of projective space.
\newblock {\em Algebra Number Theory}, 10(9):1983--2014, 2016.

\bibitem[DELV11]{DELV11}
Olivier Debarre, Lawrence Ein, Robert Lazarsfeld, and Claire Voisin.
\newblock Pseudoeffective and nef classes on abelian varieties.
\newblock {\em Compos. Math.}, 147(6):1793--1818, 2011.

\bibitem[EH16]{EH16}
David Eisenbud and Joe Harris.
\newblock {\em 3264 and all that---a second course in algebraic geometry}.
\newblock Cambridge University Press, Cambridge, 2016.

\bibitem[ELB83]{EL83}
Georges Elencwajg and Patrick Le~Barz.
\newblock Une base de {${\rm Pic}({\rm Hilb}\sp{k}\,{\bf P}\sp{2})$}.
\newblock {\em C. R. Acad. Sci. Paris S\'{e}r. I Math.}, 297(3):175--178, 1983.

\bibitem[ELB88]{EL86}
G.~Elencwajg and P.~Le~Barz.
\newblock Explicit computations in {${\rm Hilb}^3{\bf P}^2$}.
\newblock In {\em Algebraic geometry ({S}undance, {UT}, 1986)}, volume 1311 of
  {\em Lecture Notes in Math.}, pages 76--100. Springer, Berlin, 1988.

\bibitem[ES87]{ES87}
Geir Ellingsrud and Stein~Arild Str\text{\o}mme.
\newblock On the homology of the {H}ilbert scheme of points in the plane.
\newblock {\em Invent. Math.}, 87(2):343--352, 1987.

\bibitem[Fan95]{F95}
Barbara Fantechi.
\newblock Deformation of {H}ilbert schemes of points on a surface.
\newblock {\em Compositio Math.}, 98(2):205--217, 1995.

\bibitem[FG93]{FG93}
Barbara Fantechi and Lothar G\"{o}ttsche.
\newblock The cohomology ring of the {H}ilbert scheme of {$3$} points on a
  smooth projective variety.
\newblock {\em J. Reine Angew. Math.}, 439:147--158, 1993.

\bibitem[Ful98]{F98}
William Fulton.
\newblock {\em Intersection theory}, volume~2 of {\em Ergebnisse der Mathematik
  und ihrer Grenzgebiete. 3. Folge. A Series of Modern Surveys in Mathematics
  [Results in Mathematics and Related Areas. 3rd Series. A Series of Modern
  Surveys in Mathematics]}.
\newblock Springer-Verlag, Berlin, second edition, 1998.

\bibitem[Got90]{Go90}
Lothar Gottsche.
\newblock The {B}etti numbers of the {H}ilbert scheme of points on a smooth
  projective surface.
\newblock {\em Math. Ann.}, 286(1-3):193--207, 1990.

\bibitem[Hui16]{Hu16}
Jack Huizenga.
\newblock Effective divisors on the {H}ilbert scheme of points in the plane and
  interpolation for stable bundles.
\newblock {\em J. Algebraic Geom.}, 25(1):19--75, 2016.

\bibitem[Iar72]{I72}
A.~Iarrobino.
\newblock Punctual {H}ilbert schemes.
\newblock {\em Bull. Amer. Math. Soc.}, 78:819--823, 1972.

\bibitem[Jel20]{J20}
Joachim Jelisiejew.
\newblock Pathologies on the {H}ilbert scheme of points.
\newblock {\em Invent. Math.}, 220(2):581--610, 2020.

\bibitem[LQW04]{LQW04}
Wei-Ping Li, Zhenbo Qin, and Weiqiang Wang.
\newblock The cohomology rings of {H}ilbert schemes via {J}ack polynomials.
\newblock In {\em Algebraic structures and moduli spaces}, volume~38 of {\em
  CRM Proc. Lecture Notes}, pages 249--258. Amer. Math. Soc., Providence, RI,
  2004.

\bibitem[MS90]{MS90}
Raquel Mallavibarrena and Ignacio Sols.
\newblock Bases for the homology groups of the {H}ilbert scheme of points in
  the plane.
\newblock {\em Compositio Math.}, 74(2):169--201, 1990.

\bibitem[Nak96]{N96}
Hiraku Nakajima.
\newblock Jack polynomials and hilbert schemes of points on surfaces.
\newblock 1996.
\newblock Available at \url{arxiv.org/abs/alg-geom/9610021}.

\bibitem[OP10]{OP10}
A.~Okounkov and R.~Pandharipande.
\newblock Quantum cohomology of the {H}ilbert scheme of points in the plane.
\newblock {\em Invent. Math.}, 179(3):523--557, 2010.

\bibitem[Ran16]{R16}
Ziv Ran.
\newblock Incidence stratifications on {H}ilbert schemes of smooth surfaces,
  and an application to {P}oisson structures.
\newblock {\em Internat. J. Math.}, 27(1):1650006, 8, 2016.

\bibitem[RL90]{RL90}
Francesc Rossell\'{o}-Llompart.
\newblock The {C}how ring of {${\rm Hilb}^3{\bf P}^3$}.
\newblock In {\em Enumerative geometry ({S}itges, 1987)}, volume 1436 of {\em
  Lecture Notes in Math.}, pages 225--255. Springer, Berlin, 1990.

\bibitem[RLXD91]{RLXD92}
F.~Rossell\'{o}-Llompart and S.~Xamb\'{o}-Descamps.
\newblock Chow groups and {B}orel-{M}oore schemes.
\newblock {\em Ann. Mat. Pura Appl. (4)}, 160:19--40 (1992), 1991.

\bibitem[Rob88]{R88}
Joel Roberts.
\newblock Old and new results about the triangle varieties.
\newblock In {\em Algebraic geometry ({S}undance, {UT}, 1986)}, volume 1311 of
  {\em Lecture Notes in Math.}, pages 197--219. Springer, Berlin, 1988.

\bibitem[RS21]{RS21}
Tim Ryan and Alexander Stathis.
\newblock Higher codimension cycles on the {H}ilbert scheme of three points on
  the projective plane.
\newblock {\em J. Pure Appl. Algebra}, 225(10):Paper No. 106665, 26, 2021.

\bibitem[Sza21]{S21}
Micha\l{} Szachniewicz.
\newblock Non-reducedness of the hilbert schemes of few points.
\newblock 2021.
\newblock \textit{Preprint}. Available at \url{arxiv.org/abs/2109.11805}.

\end{thebibliography}

\end{document}